\documentclass[11pt,reqno]{amsart}
\usepackage{amsfonts,amssymb,amsmath,amsthm,graphicx,color,amscd,xspace,verbatim,dsfont,mathrsfs}
\usepackage{scalerel}

\usepackage{hyperref}

\newtheorem{theorem}{Theorem}[section]
\newtheorem{lemma}[theorem]{Lemma}
\newtheorem{proposition}[theorem]{Proposition}

\theoremstyle{definition}
\newtheorem{definition}[theorem]{Definition}

\newtheorem{remark}[theorem]{Remark}

\newcommand{\R}{{\mathbb R}}

\newcommand{\BB}{\mathbb{B}}

\newcommand{\tr}{\mathrm{tr}^*}

\newcommand{\be}{\small\begin{equation}}
\newcommand{\bel}[1]{\small\begin{equation}\label{#1}}
\newcommand{\ee}{\end{equation}\normalsize}

\newcommand{\BA}{\begin{array}}
\newcommand{\EA}{\end{array}}
\newcommand{\BAN}{\renewcommand{\arraystretch}{1.2}
\setlength{\arraycolsep}{2pt}\begin{array}}
\newcommand{\BAV}[2]{\renewcommand{\arraystretch}{#1}
\setlength{\arraycolsep}{#2}\begin{array}}

\newcommand{\BSA}{\begin{subarray}}
\newcommand{\ESA}{\end{subarray}}
\newcommand{\BAL}{\begin{aligned}}
\newcommand{\EAL}{\end{aligned}}

\newcommand{\forevery}{\quad \forall}

\newcommand{\norm}[1]{\left \|#1\right \|}

\newcommand{\supp}{\mathrm{supp}\,}
\newcommand{\dist}{\mathrm{dist}\,}

\newcommand{\diam}{\mathrm{diam}\,}

\newcommand{\prt}{\partial}

\newcommand{\tl}{\tilde}
\newcommand{\sbs}{\subset}

\newcommand\1{{\ensuremath {\mathds 1} }}

\def\ga{\alpha}            
             
\def\gth{\theta}                         \def\vge{\varepsilon}
\def\gf{\phi}           
            \def\gl{\lambda}
\def\gm{\mu}        \def\gn{\nu}         
            
       \def\gt{\tau}
      \def\gw{\omega}
                
     \def\Gd{\Delta}      

\def\Gl{\Lambda}          
\def\Gw{\Omega}              

   \def\CM{{\mathcal M}}   \def\CN{{\mathcal N}}
   \def\CO{{\mathcal O}}   
\def\CA{{\mathcal A}}   \def\CB{{\mathcal B}}   \def\CC{{\mathcal C}}
\def\CD{{\mathcal D}}      \def\CF{{\mathcal F}}
   \def\CH{{\mathcal H}}   
      \def\CL{{\mathcal L}}

   \def\BBB {\mathbb B}    
       
\def\BBG {\mathbb G}       
\def\BBJ {\mathbb J}   \def\BBK {\mathbb K}    
   \def\BBN {\mathbb N}    
   \def\BBR {\mathbb R}    \def\BBS {\mathbb S}

\def\GTM {\mathfrak M}

\def\tr{\mathrm{tr}}

\newcommand{\sB}{\mathscr{B}}
\newcommand{\sC}{\mathscr{C}}

\newcommand{\sE}{\mathscr{E}}

\def \dd {\,\mathrm{d}}
\def \dx {\mathrm{d}x}

\newcommand{\ei}{{\phi_{\xm }}}
\newcommand{\xa}{\alpha}
\newcommand{\xb}{\beta}
\newcommand{\xg}{\gamma}

\newcommand{\xd}{\delta}

\newcommand{\xe}{\varepsilon}

\newcommand{\xk}{\kappa}
\newcommand{\xl}{\lambda}

\newcommand{\xm}{\mu}
\newcommand{\xn}{\nu}

\newcommand{\xS}{\Sigma}
\newcommand{\xf}{\phi}

\newcommand{\xo}{\omega}
\newcommand{\xO}{\Omega}

\newcommand{\myint}[2]{{\displaystyle \int_{#1}^{#2}}}

\def\Nthb{\BBN_{\xa}}

\newcommand{\ap}{{\xa_{\scaleto{+}{3pt}}}}
\newcommand{\am}{{\xa_{\scaleto{-}{3pt}}}}

\def\bal#1\eal{\small\begin{align*}#1\end{align*}\normalsize}
\def\ba#1\ea{\small\begin{align}#1\end{align}\normalsize}
\numberwithin{equation}{section}

\title[Elliptic Schr\"odinger equations]{Elliptic Schr\"odinger equations \\ with gradient-dependent nonlinearity \\ and  Hardy potential singular on manifolds}

\author[K. T. Gkikas]{Konstantinos T. Gkikas}
\address{Konstantinos T. Gkikas, Department of Mathematics, University of the Aegean,\newline 83200 Karlovassi, Samos,
Greece\newline
Department of Mathematics, National and Kapodistrian University of Athens, 15784 Athens, Greece.}
\email{kgkikas@aegean.gr}

\author[P.T. Nguyen]{Phuoc-Tai Nguyen}
\address{Phuoc-Tai Nguyen, Department of Mathematics and Statistics, Masaryk University, Brno, Czech Republic}
\email{ptnguyen@math.muni.cz}

\date{\today}

\begin{document}

\begin{abstract}
Let $\Omega \subset \R^N$ ($N \geq 3$) be a $C^2$ bounded domain and $\Sigma \subset \Omega$ is a $C^2$ compact boundaryless submanifold in $\mathbb{R}^N$  of dimension $k$, $0\leq k < N-2$. For $\mu\leq (\frac{N-k-2}{2})^2$, put $L_\mu := \Delta + \mu d_{\Sigma}^{-2}$ where $d_{\Sigma}(x) = \dist(x,\Sigma)$. We study boundary value problems for equation $-L_\mu u = g(u,|\nabla u|)$ in $\Omega \setminus \Sigma$, subject to the boundary condition $u=\nu$ on $\partial \Omega \cup \Sigma$, where $g: \mathbb{R} \times \mathbb{R}_+ \to \mathbb{R}_+$ is a continuous and nondecreasing function with $g(0,0)=0$, $\nu$ is a given nonnegative measure on $\partial \Omega \cup \Sigma$. When $g$ satisfies a so-called subcritical integral condition, we  establish an existence result for the problem under a smallness assumption on $\nu$. If $g(u,|\nabla u|) = |u|^p|\nabla u|^q$, there are  ranges of $p,q$, called subcritical ranges, for which the subcritical integral condition is satisfied, hence the problem admits a solution. Beyond these ranges, where the  subcritical integral condition may be violated, we establish various criteria on $\nu$ for the existence of a solution to the problem expressed in terms of appropriate Bessel capacities. 
	
\medskip
	
	
\medskip
	
	
\end{abstract}

\maketitle
\tableofcontents
\section{Introduction}
Let $\Omega \subset \R^N$ ($N \geq 3$) be a $C^2$ bounded domain and $\Sigma \subset \Omega$ is a $C^2$ compact boundaryless submanifold of dimension $k$, $0\leq k < N-2$. When $k = 0$ we assume that $\Sigma = \{0\} \subset \Omega$. Denote $d(x)=\dist(x,\partial \Omega)$ and $d_{\Sigma}(x) = \dist(x,\Sigma)$. For $\mu \in \R$, let $L_\mu$ be the Schr\"odinger operator defined by 
$$L_\mu := \Delta + \mu d_\Sigma^{-2}.$$ 
In this paper, we study the semilinear equation with a \textit{source nonlinear term depending on solutions and their gradient} of the form
\be \label{eq:ugradu}
-L_\mu u = g (u,|\nabla u| ) \quad \text{in } \Gw\setminus\Sigma.
\ee

In the free-potential case, namely $\Sigma=\emptyset$ and $\mu=0$, this equation and associated boundary value problems, as well as other related problems, have been investigated in numerous papers. Among many other results, we refer to \cite{ChiCir_2015,BHV,MMV_2019,FPS_2020} for $g(u,|\nabla u|) = |u|^p|\nabla u|^q$ and \cite{MMV_2020} for  $g(u,|\nabla u|) = |u|^p + M|\nabla u|^q$.

When $\mu \neq 0$, the existence and sharp two-sided estimates of the Green function $G_\mu$ and Martin kernel $K_\mu$ of $L_\mu$ were established in  \cite{GkiNg_linear}, which are important ingredients in the theory for inhomogeneous linear equation associated to \eqref{eq:ugradu}. The case where $g$ depends only on $u$ was treated in \cite{GkiNg_source} where various sharp necessary and sufficient conditions for the existence of solutions 
were derived. Relevant results can be also found in \cite{CheZho_2018} for the case $\Sigma=\{0\} \subset \Omega$, in \cite{FelTer_2006, CheChe_2023} for multi-singular inverse-square potentials, in  \cite{DD1,DD2,DN,Fal_2013} for existence and nonexistence results regarding the source nonlinearity $g(u)=u^p$ in different frameworks, and in \cite{GkiNg_absorp} for absorption nonlinearities. 

When $\mu \neq 0$ and $g$ depends on both $u$ and $\nabla u$, equation \eqref{eq:ugradu} becomes more intricate, governed by the following competing factors.

$\bullet$ The Schr\"odinger operator: In this operator, the Hardy potential $\mu d_{\Sigma}^{-2}$ is singular on $\Sigma$ and enjoys the same scaling as the Laplacian $\Delta$, hence it plays a far more significant role than a mere perturbation of the Laplacian. If $\mu \leq C_{\Omega,\Sigma}$, where $C_{\Omega,\Sigma}$ is the best Hardy constant defined by
\bal \mathcal{C}_{\Omega,\Sigma}:=\inf_{0 \not \equiv \varphi \in H^1_0(\Omega)}\frac{\int_\Omega |\nabla \varphi|^2\dx}{\int_\Omega d_\Sigma^{-2}\varphi^2 \dx},
\eal 
then $L_\mu$ is positive defined. In case of the critical potential, namely $\mu=C_{\Omega,\Sigma}$, $L_\mu$ is not uniformly elliptic, which leads to the invalidity of classical techniques. 

$\bullet$ The domain $\Omega$ and the manifold $\Sigma$: The geometric properties of $\Omega$ and $\Sigma$ determine the value of the best Hardy constant $\mathcal{C}_{\Omega,\Sigma}$. In particular, if $\Sigma=\{0\}$ then  ${\mathcal C}_{\Omega,\{0\}}=\left(\frac{N-2}{2} \right)^2$. In general, it is well-known that ${\mathcal C}_{\xO,\Sigma}\in (0,H^2]$ (see e.g. \cite{BFT, DD1, DD2}), where
\be \label{valueH}
H:=\frac{N-k-2}{2}.
\ee
Moreover, ${\mathcal C}_{\Omega,\Sigma}=H^2$ if $-\Delta d_\Sigma^{2+k-N} \geq 0$ in the sense of distributions in $\Omega \setminus \Sigma$ or if $\Omega=\Sigma_\beta$ with $\beta$  small enough (see \cite{BFT}), where
\bal
\Sigma_\beta :=\{ x \in \R^N \setminus \Sigma: d_\Sigma(x) < \beta \}.
\eal   

$\bullet$ The nonlinear term: Since $g$ depends on solutions and their gradient, equation \eqref{eq:ugradu} cannot be treated by variational approach nor monotonicity method. Moreover, the competition between $u$ and $|\nabla u|$ makes it more complicated to handle the nonlinear term.

$\bullet$ Boundary data: Since the Hardy potential $d_\Sigma^{-2}$ blows up near $\Sigma$ and is bounded near $\partial \Omega$, to investigate boundary value problems for \eqref{eq:ugradu}, boundary conditions should be imposed in an appropriate sense. The effect of the potential would be better understood if the concentration of the boundary data is considered separately on $\partial \Omega$ and on $\Sigma$.

The combined effects of the above features gives rise to various ranges for the existence of solutions to boundary value problems for \eqref{eq:ugradu}, complicating the analysis and requiring refined techniques to deal with these ranges. \medskip

For $\mu \leq H^2$, let $\am$ and $\ap$ be the roots of the algebraic equation $\ga^2 - 2H\ga + \mu=0$, i.e.
\bel{apm}
\am:=H-\sqrt{H^2-\mu}, \quad \ap:=H+\sqrt{H^2-\mu}.
\ee
Note that $\am\leq H\leq\ap<2H$, and $\am \geq 0$ if and only if $\mu \geq 0$. 

For $\mu \leq H^2$, the first eigenvalue of $-L_\mu$ is given by 
\bal \lambda_\mu:=\inf\left\{\int_{\Gw}\left(|\nabla u|^2-\mu d_\Sigma^{-2}u^2\right)dx: u \in C_c^1(\Omega), \| u \|_{L^2(\Omega)}=1\right\}.
\eal
Note that $\lambda_\mu>-\infty$ due to \cite[Lemma 2.4 and Theorem 2.6]{DD1} (see also \cite{DD2}). The corresponding eigenfunction $\phi_\mu$, normalized by $\| \phi_\mu \|_{L^2(\Omega)}=1$, satisfies two-sided estimate $\phi_\mu \asymp d\,d_\Sigma^{-\am}$ in $\Omega \setminus \Sigma$ (see Appendix \ref{subsect:eigen} for more detail). We remark that $\lambda_\mu>0$ if $\mu<\CC_{\Omega,\Sigma}$; however, this does not hold true in general.  

Under the assumption $\lambda_\mu>0$, the existence and sharp two-sided estimates of the Green function $G_\mu$ and Martin kernel $K_\mu$ associated to $-L_\mu$ were established in \cite{GkiNg_linear}. \textit{Let us assume throughout the paper that}
\be \label{assump1}
\mu \leq H^2 \quad \text{and} \quad \lambda_\mu > 0.
\ee

Put
\bal
\GTM(\Omega \setminus \Sigma;\phi_\mu) := \{ \sigma \text{ is a Radon measure on } \Omega \setminus \Sigma:  \| \sigma\|_{\GTM(\Omega \setminus \Sigma;\ei)}:=\int_{\Omega \setminus \Sigma}\phi_\mu \dd |\sigma|<\infty \}, \\ \GTM(\partial \Omega \cup \Sigma) := \{ \lambda \text{ is a Radon measure on } \partial \Omega \cup \Sigma:  \| \lambda\|_{\GTM(\partial \Omega \cup \Sigma)}:=\lambda(\partial \Omega \cup \Sigma) < \infty\},
\eal
and denote by $\GTM^+(\Omega \setminus \Sigma;\phi_\mu)$ and $\GTM^+(\partial \Omega \cup \Sigma)$ their positive cones respectively.
Let $\BBG_{\mu}$ and $\BBK_{\mu}$ be the Green operator and Martin operator respectively defined as
\begin{align} \label{BBG} 
	\BBG_{\mu}[\tau](x) &:= \int_{\Omega \setminus \Sigma} G_\mu(x,y)d\tau(y), \quad x \in \Omega \setminus \Sigma, \quad \tau \in \GTM(\Omega \setminus \Sigma;\phi_\mu), \\ \label{BBK}
	\BBK_{\mu}[\nu](x) &:= \int_{\partial \Omega \cup \Sigma} K_\mu[\nu](x,y)d\nu(y), \quad x \in \Omega \setminus \Sigma, \quad \nu \in \GTM(\partial \Omega \cup \Sigma).  
\end{align}


To formulate boundary value problems for \eqref{eq:ugradu}, we will employ the notion of boundary trace introduced in \cite{GkiNg_linear}. The notion is defined by normalizing the behavior of functions near $\Sigma$ with respect to a suitable weight behaving as $d_{\Sigma}^{-\ap}$, hence allowing us to capture possible singularity phenomena on $\Sigma$. 
\begin{definition}[Boundary trace] \label{nomtrace}
We say that	a function $u\in W^{1,\kappa}_{\mathrm{loc}}(\xO\setminus\xS),$ for some $\kappa>1,$ possesses a \emph{boundary trace}, denoted by $\tr(u)$,  if there exists a measure $\nu \in\GTM(\partial \Omega \cup \Sigma)$ such that for any $C^2$ exhaustion  $\{ O_n \}$ of $\Omega \setminus \Sigma$, there  holds
	\bal
	\lim_{n\rightarrow\infty}\int_{ \partial O_n}\phi u\dd\omega_{O_n}^{x_0}=\int_{\partial \Omega \cup \Sigma} \phi \dd\nu \quad\forall \phi \in C(\overline{\Omega}).
	\eal
	In this case, we will write $\tr(u)=\nu$.
\end{definition}
\noindent More details about the $L_\mu$-harmonic measure $\omega_{O_n}^{x_0}$ is presented in Appendix \ref{appendix:har-measure}.

%

With this notion of boundary trace, the boundary value problem for \eqref{eq:ugradu} reads as
\be \label{sourceprobrho} 
	\left\{ \begin{aligned} 
		-L_\mu u &= g (u,|\nabla u| ) \quad \text{in } \Gw\setminus\Sigma, \\ 
		\tr(u) &= \varrho \nu,
	\end{aligned} \right. 
\ee
where $\varrho$ is a positive parameter and $\nu \in \GTM^+(\partial \Omega \cup \Sigma)$ with $\| \nu \|_{\GTM(\partial \Omega \cup \Sigma)}=1$. Here the boundary datum is written in the form $\varrho \nu$ with $\| \nu \|_{\GTM(\partial \Omega \cup \Sigma)}=1$ in stead of $\tilde \nu \in \GTM^+(\partial \Omega \cup \Sigma)$ to make the exhibition of the paper neater and does not cause any loss of generality. Weak solutions to problem \eqref{sourceprobrho} are defined as follows:
\begin{definition}
We say that $u$ is a \textit{weak solution} of problem \eqref{sourceprobrho} if $u\in L^1(\Omega;\ei)$, $g(u,|\nabla u|) \in L^1(\Omega;\ei)$ and $u$ satisfies
	\bal
	- \int_{\Gw}u L_{\xm }\zeta \dd x = \int_{\Gw \setminus \Sigma} g(u,|\nabla u|) \zeta \dd x - \varrho \int_{\Gw} \mathbb{K}_{\xm}[\xn]L_{\xm }\zeta \dd x
	\qquad\forall \zeta \in\mathbf{X}_\xm(\xO\setminus \Sigma),
	\eal
	where the space of test functions is
	\bal {\bf X}_\mu(\Gw\setminus \Sigma):=\{ \zeta \in H_{\rm loc}^1(\Omega \setminus \Sigma): \phi_\mu^{-1} \zeta \in H^1(\Gw;\phi_\mu^{2}), \, \phi_\mu^{-1}L_\mu \zeta \in L^\infty(\Omega)  \}.
	\eal
\end{definition}
We note that a function $u$ is a weak solution to \eqref{sourceprobrho} if and only if it can be decomposed as
\bal u = \BBG_{\mu}[g(u,|\nabla u|)] + \varrho \BBK_{\mu}[\nu] \quad \text{ a.e. in }  \Omega \setminus \Sigma.
\eal

Put
\ba \label{p*q*}
p_*:=\min\left\{\frac{N-\am}{N-\am-2},\frac{N+1}{N-1}\right\} \quad \text{and} \quad 
q_*:=\min\left\{\frac{N-\am}{N-\am-1}, \frac{N+1}{N}\right\}.
\ea
Notice that if $\mu>0$ then $\am>0$, hence $p_*=\frac{N+1}{N-1}$ and $q_*=\frac{N+1}{N}$. These two exponents play an important role in the study of \eqref{sourceprobrho}, as shown below.

\begin{theorem}\label{exist1}  
Assume $g$ satisfies
	\bel{integralcond} \Gl_{g}:=\int_{1}^\infty g(s,s^{\frac{p_*}{q_*}})s^{-1-p_*}ds <\infty,
	\ee
	\bel{multi}
	g(as,bt) \leq \tl k (a^{\tl p} +  b^{\tl q})g(s,t) \quad \forall (a,b,s,t) \in \BBR_+^4,
	\ee
	for some $\tl p>1$, $\tl q>1$, $\tl k > 0$.
	Then there exists $\varrho_0>0$ depending on $N,\mu,\Gw,\Gl_g,\tl k, \tl p, \tl q$ such that for any $\varrho \in (0,\varrho_0)$, problem \eqref{sourceprobrho} admits a positive weak solution $u \geq \varrho \BBK_\mu [\nu]$ a.e. in $\Omega \setminus \Sigma$.
\end{theorem}
Condition \eqref{integralcond} is commonly called a \textit{subcritical integral condition} which describes a sharp growth of $g$ in its two arguments for the solvability of problem \eqref{sourceprobrho}. Let us point out that the existence result is obtained by using the Schauder fixed point theorem. However, due to the peculiar effect of the Hardy potential and the presence of the gradient-dependent nonlinearity, the proof is non-trivial, necessitating to establish sharp bounds on the gradient of the Green function and Martin kernel in weak-$L^p$ spaces (see Theorems \ref{lpweakgrangreen}, \ref{lpweakgrangreen2} and \ref{lpweakgranmartin1}), and to be implemented through an approximation procedure in which assumptions \eqref{integralcond} and \eqref{multi} are employed. 

Two typical examples of $g$ to be kept in mind are presented below.

$\bullet$ When $g(s,t) =|s|^p + t^q$ for $s \in \R$, $t \in \R_+$, $p>1$, $q>1$ then $g$ satisfies \eqref{integralcond} and \eqref{multi} if $1<p<p_*$ and $1<q<q_*$.

$\bullet$ When $g(s,t) =|s|^pt^q$ for $s \in \R$, $t \in \R_+$, $p \geq 0$, $q \geq 0$, $p+q>1$ then $g$ satisfies \eqref{integralcond} and \eqref{multi} if
\begin{equation} \label{subcriticalrange}
q_* p + p_* q < p_* q_*.
\end{equation}  

\noindent Whenever $p,q$ satisfy the above conditions, we say that $p,q$ are in a \textit{subcritical range}. 

Between the above examples, we will focus on the  case   $g(u,|\nabla u|) =|u|^p|\nabla u|^q$  in which the interaction of $u^p$ and $|\nabla u|^q$ is more intricate and interesting. A sufficient condition for the existence of a solution can be expressed in terms of appropriate capacities $\text{Cap}_{\BBN_{\alpha,\sigma},s}^{b,\theta}(E)$ (see \eqref{dualcap} for the definition of these capacities below).

We will first consider the case $\xm<\frac{(N-2)^2}{4}$.

\begin{theorem}\label{existth}
	Let $\xm<\frac{(N-2)^2}{4}$ and  $p \geq 0$, $q\geq0$ such that $p+q>1$ and
	$\am p+(\am+1)q<2+\am$. Assume  $\nu \in \GTM^+(\partial \Omega \cup \Sigma)$ satisfies
	\ba\label{con3-a}
	\xn(E)\leq C\, \mathrm{Cap}_{\BBN_{2\am+1,1},(p+q)'}^{p+1,-\am(p+1)-(\am+1)q}(E) \quad \text{for any Borel set } E \subset \overline{\xO}.
	\ea
	Then for $\varrho>0$ small enough, the problem
	\bel{problem:source-power} \left\{ \BAL -L_\mu u &=|u|^p|\nabla u|^q \quad \text{in } \Gw\setminus\xS, \\ \tr(u) &= \varrho \nu,
	\EAL \right.
	\ee
	admits a nonnegative weak solution.
\end{theorem}

It is noteworthy that condition \eqref{con3-a} may hold for $p,q$ beyond the subcritical range \eqref{subcriticalrange} even if the condition \eqref{integralcond} is violated. While the proof of Theorem \ref{existth} also utilizes the Schauder fixed point theorem, its core approach differs significantly from that of Theorem \ref{exist1}. A key ingredient here is an existence result in Theorem \ref{sourceth}  which relies on a careful and nontrivial adaptation of an abstract result by Kalton and Verbitsky \cite{KV} to our framework under the condition $\am p + (\am+1)q < 2 + \am$.

It is natural to ask whether condition \eqref{con3-a} could be simplified. The next result provides \textit{subcritical ranges} of $p,q$  for the existence of a solution of \eqref{problem:source-power} according to whether $\nu$ is concentrated on $\Sigma$ or on $\partial \Omega$. This is achieved by applying Theorem \ref{existth} and using equivalent criteria in Theorem \ref{sourceth}, taking into account particular critical exponents corresponding to the concentration of $\nu$ in Theorem \ref{lpweakgranmartin1}.

\begin{proposition} \label{th:existence-source-3}
Let $\xm<\frac{(N-2)^2}{4}$, $\xn\in\mathfrak{M}^+(\partial\xO\cup\xS)$, and $p \geq 0$, $q\geq0$ such that $p+q>1$ and $\am p+(\am+1)q<2+\am$. Assume one of the following conditions holds.
	
	(i) $\xn$ is compactly supported on $\xS$ and $(N-\am-2)p+(N-\am-1)q<N-\am$.

	(ii) $\xn$ is compactly supported on $\partial\xO$ and $(N-1)p+Nq<N+1$.

\noindent Then for $\varrho>0$ small enough, problem \eqref{problem:source-power} admits a nonnegative weak solution.	
\end{proposition}

We remark that the subcritical range in (i) coincides \eqref{subcriticalrange} with $p_*=\frac{N-\am}{N-\am-2}$ and $
q_*=\frac{N-\am}{N-\am-1}$, while the the subcritical range in (ii) coincides \eqref{subcriticalrange} with $p_*=\frac{N+1}{N-1}$ and $
q_*=\frac{N+1}{N}$. The task of showing existence results beyond the above ranges is more challenging. To do this, we introduce appropriate Bessel capacities restricted on $\Sigma$ and $\partial \Omega$ (see \eqref{Capsub} for the definition of such capacities $\mathrm{Cap}_{\gth,\kappa}^{\Gamma}$) which measure subsets of $\Sigma$ or $\partial \Omega$ in a subtle manner. These capacities are equivalent to the capacities used in Theorem \ref{existth} thanks to the delicate two-sided estimate in Lemma \ref{besov} (which holds under condition \eqref{3condition}). Consequently, we can employ Theorem \ref{existth} to obtain a sufficient condition \eqref{cap} in case $\nu$ is supported on $\Sigma$ or condition \eqref{cond:cap-2} in case $\nu$ is supported on $\partial \Omega$.

\begin{theorem} \label{th:existence-source-4}
	Let $k\geq1,$ $p,q\geq0$ be such that $p+q>1$ and
	\bel{3condition} \left\{ \BAL
	&\am p+(\am+1)q<N-k-\am,\\
	&\ap p+(\ap+1)q<2+\ap,\\
	&(N-\am-2)p+(N-\am-1)q>N-k-\am.
	\EAL \right. \ee
Assume that $\xn\in \mathfrak{M}^+(\partial\xO\cup\xS)$ is compactly supported on $\xS$ and satisfies
\bel{cap}
\nu(E) \leq C\mathrm{Cap}_{\vartheta,(p+q)'}^{\xS}(E)\quad\text{for any Borel set}\;E\subset\xS,
\ee
where
\bel{vartheta}
	\vartheta:=\frac{2+\ap -\ap p - (\ap+1)q}{p+q} .
\ee
Then for $\varrho>0$ small enough, problem \eqref{problem:source-power} admits a nonnegative weak solution.	
\end{theorem}


\begin{theorem} \label{th:existence-source-5}
Let $\xm<\frac{(N-2)^2}{4}$, $p \geq 0$, $q\geq0$ such that $p+q>1,$ $q<2$ and $\am p+(\am+1)q<2+\ap$.
Assume  $\xn\in \mathfrak{M}^+(\partial\xO\cup\xS)$ is compactly supported on $\partial\xO$ and satisfies
	\bel{cond:cap-2}
	\nu(E) \leq C\mathrm{Cap}_{\frac{2-q}{p+q},(p+q)'}^{\partial\xO}(E)\quad\text{for any Borel set}\;E\subset\partial\xO.
	\ee
Then for $\varrho>0$ small enough, problem \eqref{problem:source-power} admits a nonnegative weak solution.	
\end{theorem}

It is worth pointing out that conditions \eqref{cap} and \eqref{cond:cap-2} are fulfilled if $p,q$ satisfy respectively (i) and (ii) in Proposition \ref{th:existence-source-3}. We also notice that conditions \eqref{cap} and \eqref{cond:cap-2} are more general than \cite[condition 4 in Theorem 1.5 and condition 4 in Proposition 1.6]{GkiNg_source} respectively. 
 
The case $\xS=\{0\}$ and $\xm= \frac{(N-2)^2}{4}$ is treated separately, involving a perturbation. A counter part of Theorem \ref{existth} in this case is stated below.
\begin{theorem}\label{existthcr}
Let $\xS=\{0\}$, $\xm= \frac{(N-2)^2}{4}$, $p\geq 0$, $q\geq0$ such that $p+q>1$ and $(N-2)p+Nq<N+2$, and $0<\varepsilon <\min(N-2,2,\frac{N+2- (N-2)p - Nq}{2(p+q)})$. Assume $\xn\in\mathfrak{M}^+(\partial\xO\cup \{0\})$ satisfies
	\ba \label{H2-cap-1-0} 
	\xn(E)\leq C\, \mathrm{Cap}_{\BBN_{N-1-\xe,1},(p+q)'}^{p+1,-\frac{N-2}{2}(p+1)-\frac{N}{2}q}(E) \quad \text{for any Borel set } E \subset \overline{\xO}.
	\ea
	Then for $\varrho>0$ small enough, problem \eqref{problem:source-power} admits a nonnegative weak solution (see \eqref{dualcap} for the definition of capacities $\text{Cap}_{\BBN_{\alpha,\sigma},s}^{b,\theta}(E)$).
	
\end{theorem}

The proof of Theorem \ref{existthcr} relies on equivalent criteria for existence results in Theorem \ref{sourceth2} for which assumption on $\varepsilon$ is needed.

When $\nu$ is concentrated at $0$ or on $\partial \Omega$, subcritical ranges of $p,q$ for the existence of a weak solution to problem \eqref{problem:source-power} are described in the next two theorems.

\begin{proposition} \label{th:existence-H2-2}
Let $\xS=\{0\}$, $\xm= \frac{(N-2)^2}{4}$, $p\geq 0$, $q\geq0$ such that  $p+q>1$ and $(N-2)p+Nq<N+2$. Assume one of the following conditions holds
	
	(i) $\nu=\delta_0$, the Dirac measure concentrated at $0$,
	
	(ii) $\xn$ is compactly supported on $\partial\xO$ and
	$(N-1)p+Nq<N+1$.
	
	\noindent Then for $\varrho>0$ small enough, problem \eqref{problem:source-power} admits a nonnegative weak solution.
\end{proposition}

\begin{theorem} \label{th:existence-H2-3}
Let $\xS=\{0\}$, $\xm= \frac{(N-2)^2}{4}$, $p,q\geq0$ such that $p+q>1$, $q<2$ and $(N-2)p+Nq<N+2$. Assume $\xn\in \mathfrak{M}^+(\partial\xO\cup \{0\})$ is compactly supported on $\partial\xO$ and satisfies
	\bal 
	\nu(E) \leq C\mathrm{Cap}_{\frac{2-q}{p+q},(p+q)'}^{\partial\xO}(E)\quad\text{for any Borel set}\;E\subset\partial\xO.
	\eal
Then for $\varrho>0$ small enough, problem \eqref{problem:source-power} admits a nonnegative weak solution.
\end{theorem}

We note that when $\mu=(\frac{N-2}{2})^2$, then $\am=\frac{N-2}{2}$, hence both conditions $\am p+(\am+1)q<2+\am$ and $(N-\am-2)p+(N-\am-1)q<N-\am$ in Theorem \ref{th:existence-source-3} reduce to $(N-2)p+Nq<N+2$. Therefore, Theorem \ref{th:existence-H2-3} is in accordance with Theorem \ref{th:existence-source-3}, and both of them together provide the subcritical ranges of $p,q$ \textit{for any} $\mu \leq (\frac{N-2}{2})^2$. Similarly, the existence result in supercritical ranges of $p,q$ when $\nu$ is concentrated on $\partial \Omega$ holds for any $\mu \leq (\frac{N-2}{2})^2$ thanks to Theorem \ref{th:existence-source-5} and Theorem \ref{th:existence-H2-3}.

Let us give some remark regarding the comparison between our results and relevant results in the literature. Theorem \ref{exist1} generalizes \cite[Theorem 1.3 (iii) and (iv)]{GkiNg_source} and is a counterpart of  \cite[Theorem 1.9]{GkiNg_2020} and \cite[Theorem 1.3]{GP_2024}. Moreover, Theorems \ref{th:existence-source-4} and  \ref{th:existence-source-5} cover \cite[Theorem 1.5 and Proposition 1.6 with $\tau=0$]{GkiNg_source}. In addition, our existence results are new even in case $\Sigma=\{0\}$. 

\medskip

\noindent \textbf{Organization of the paper.} In Section \ref{pre}, we give the geometric properties of $\Sigma$ and recall some background related to  the Green function and the Martin kernel of $-L_\mu$. In Section \ref{sec:weakLp}, we prove estimates on the gradient of the Green function and the Martin kernel in weighted weak Lebesgue spaces. In section \ref{subcritical source}, we give the proof of Theorem \ref{exist1}. The proofs of all the existence results for the case $g(u,|\nabla u|)=|u|^p|\nabla u|^q$ (Theorems \ref{existth}--\ref{th:existence-H2-3}) are provided in Section \ref{supercritical source}. Finally, we recall briefly the properties of the first eigenpair of $-L_\mu$ in Appendix \ref{subsect:eigen} and the definition of $L_\mu$-harmonic measure in Appendix \ref{appendix:har-measure}.

\medskip

\noindent \textbf{Acknowledgement.} K. T. Gkikas acknowledges support by the Hellenic Foundation for Research and Innovation
(H.F.R.I.) under the “2nd Call for H.F.R.I. Research Projects to support Post-Doctoral Researchers” (Project Number: 59).
 P.-T. Nguyen was supported by Czech Science Foundation, Project GA22-17403S. Part of the work was carried out during the visit of P.-T. Nguyen to the Vietnam Institute for Advanced Study in Mathematics, Vietnam, and the University of the Aegean, Greece. P-T. Nguyen gratefully acknowledges the institutions' for the hospitality. 
 
\section{Preliminaries} \label{pre}

\subsection{Notations} \label{subsec:notations} We list below notations that are frequently used in the paper.

$\bullet$ Let $\phi$ be a positive continuous function in $\Omega \setminus \Sigma$ and $\kappa \geq 1$. Let $L^\kappa(\Omega;\phi)$ be the space of functions $f$ such that
\bal \| f \|_{L^\kappa(\Omega;\phi)} := \left( \int_{\Omega} |f|^\kappa \phi \, \dd x \right)^{\frac{1}{\kappa}}.
\eal
The weighted Sobolev space $H^1(\Omega;\phi)$ is the space of functions $f \in L^2(\Omega;\phi)$ such that $\nabla f \in L^2(\Omega;\phi)$. This space is endowed with the norm
\bal \| f \|_{H^1(\Omega;\phi)}^2= \int_{\Omega} |f|^2 \phi \,\dd x +  \int_{\Omega} |\nabla f|^2 \phi \,\dd x.
\eal
The closure of $C_c^\infty(\Omega)$ in $H^1(\Omega;\phi)$ with respect to the norm $\| \cdot \|_{H^1(\Omega;\phi)}$ is denoted by $H_0^1(\Omega;\phi)$.



$\bullet$ For $\beta>0$, $ \Omega_{\beta}=\{ x \in \Omega: d(x) < \beta\}$, $\Sigma_{\beta}=\{ x \in \R^N \setminus \Sigma:  d_\Sigma(x)<\beta \}$.

$\bullet$ We denote by $c,c_1,C...$ the constant which depend on initial parameters and may change from one appearance to another.

$\bullet$ The notation $A \gtrsim B$ (resp. $A \lesssim B$) means $A \geq c\,B$ (resp. $A \leq c\,B$) where the implicit $c$ is a positive constant depending on some initial parameters. If $A \gtrsim B$ and $A \lesssim B$, we write $A \asymp B$. Throughout the paper, most of the implicit constants depend on some (or all) of the initial parameters such as $N,\Omega,\Sigma,k,\mu$ and we will omit these dependencies in the notations (except when it is necessary).

$\bullet$ For $a,b \in \BBR$, denote $a \wedge b = \min\{a,b\}$, $a \lor b =\max\{a,b \}$.

$\bullet$ For a set $D \subset \R^N$, $\1_D$ denotes the indicator function of $D$.

\subsection{Submanifold $\Sigma$.} \label{assumptionK} Throughout this paper, we assume that $\Sigma \subset \Omega$ is a $C^2$ compact submanifold in $\mathbb{R}^N$ without boundary, of dimension $k$, $0\leq k < N-2$. When $k = 0$ we assume that $\Sigma = \{0\}$.

For $x=(x_1,...,x_k,x_{k+1},...,x_N) \in \R^N$, we write $x=(x',x'')$ where $x'=(x_1,..,x_k) \in \R^k$ and $x''=(x_{k+1},...,x_N) \in \R^{N-k}$. For $\beta>0$, we denote by $B_{\beta}^k(x')$ the ball  in $\R^k$ with center at $x'$ and radius $\beta.$ For any $\xi\in \Sigma$, we denote
\bal  \Sigma_\beta &:=\{ x \in \R^N \setminus \Sigma: d_\Sigma(x) < \beta \}, \\
V(\xi,\xb)&:=\{x=(x',x''): |x'-\xi'|<\beta,\; |x_i-\Gamma_i^\xi(x')|<\xb,\;\forall i=k+1,...,N\},
\eal
for some functions $\Gamma_i^\xi: \R^k \to \R$, $i=k+1,...,N$.

The assumption that $\Sigma$ is a $C^2$ compact submanifold in $\mathbb{R}^N$ without boundary implies the existence of a constant $\xb_0>0$ such that the followings hold.
\begin{itemize}
\item[(i)] $\Sigma_{6\beta_0}\Subset \Omega$ and for any $x\in \Sigma_{6\beta_0}$, there is a unique $\xi \in \Sigma$  satisfies $|x-\xi|=d_\Sigma(x)$.

\item[(ii)] $d_\Sigma \in C^2(\Sigma_{4\beta_0})$, $|\nabla d_\Sigma|=1$ in $\Sigma_{4\beta_0}$ and there exists $\eta\in L^\infty(\Sigma_{4\beta_0})$ such that (see \cite[Lemma 2.2]{Vbook} and \cite[Lemma 6.2]{DN})
\bal
\Delta d_\Sigma(x)=\frac{N-k-1}{d_\Sigma(x)}+ \eta(x) \quad \text{for } x \in \Sigma_{4\beta_0} .
\eal

\item[(iii)] For any $\xi \in \Sigma$, there exist $C^2$ functions $\Gamma_i^\xi \in C^2(\R^k;\R)$, $i=k+1,...,N$, such that (upon relabeling and reorienting the coordinate axes if necessary), for any $\beta \in (0,6\beta_0)$, we have $V(\xi,\beta) \subset \Omega$ and
\bal
V(\xi,\beta) \cap \Sigma=\{x=(x',x''): |x'-\xi'|<\beta,\;  x_i=\Gamma_i^\xi (x'), \; \forall i=k+1,...,N\}.
\eal

\item[(iv)] There exist $m_0$ points $\xi^1, \xi^2, \ldots \xi^{m_0} \in \Sigma$ and $\beta_1 \in (0, \beta_0)$ such that
\be \label{cover}
\Sigma_{2\xb_1}\subset \cup_{j=1}^{m_0} V(\xi^j,\beta_0)\Subset \Omega.
\ee
\end{itemize}

Set
\bal \xd_\Sigma^\xi(x):=\left(\sum_{i=k+1}^N|x_i-\Gamma_i^\xi(x')|^2\right)^{\frac{1}{2}}, \qquad x=(x',x'')\in V(\xi,4\beta_0).\eal

Then there exists a constant $C=C(N,\Sigma)>0$ such that
\be\label{propdist}
d_\Sigma(x)\leq	\xd_\Sigma^{\xi^j}(x)\leq C \| \Sigma \|_{C^2} d_\Sigma(x),\quad \forall x\in V(\xi^j,2\beta_0),
\ee
where $\xi^j=((\xi^j)', (\xi^j)'') \in \Sigma$, $j=1,...,m_0$, are the points in \eqref{cover} and
\bal
\| \Sigma \|_{C^2}:=\sup\{  || \Gamma_i^{\xi^j} ||_{C^2(B_{5\beta_0}^k((\xi^j)'))}: \; i=k+1,...,N, \;j=1,...,m_0 \} < \infty.
\eal
We note that $\beta_1$ can be chosen small enough such that for any $x \in \Sigma_{\beta_1}$, we have $B(x,\beta_1) \subset V(\xi,\beta_0)$, 
where $\xi \in \Sigma$ satisfies $|x-\xi|=d_\Sigma(x)$.

\begin{lemma}[{\cite[Lemma A.1]{GkiNg_source}}] \label{lemapp:1}
Assume $\ell_1>0$, $\ell_2>0$, $\alpha_1$ and $\alpha_2$ such that $N-k+\alpha_1 + k\alpha_2 >0$. For $y \in \Omega \setminus \Sigma$, put
$
 \CA(y):= \{ x \in (\Omega \setminus \Sigma): d_{\Sigma}(x) \leq \ell_1 \quad \text{and} \quad |x-y| \leq \ell_2 d_{\Sigma}(x)^{\alpha_2} \}.
$
Then
\bal
\int_{\CA(y) \cap \Sigma_{\beta_1}} d_{\Sigma}(x)^{\alpha_1}\dx \lesssim \ell_1^{N-k+\alpha_1 + k\alpha_2}\ell_2^k.
\eal
\end{lemma}

\subsection{Estimates on the Green function and Martin kernel} \label{sec:GreenMartin} In this paper, we always assume that \eqref{assump1} holds. Let $G_\mu$ and $K_{\mu}$ be the Green kernel and Martin kernel of $-L_\mu$ in $\Omega \setminus \Sigma$ respectively.  The following results were proved in \cite[Proposition 4.1 and Theorem 1.2]{GkiNg_linear}.

\begin{proposition}[Sharp two-sided estimates on Green function]  \label{Greenkernel} ~~
	
	(i) If $\mu< \left( \frac{N-2}{2}\right)^2$ then, for any $x,y \in \Omega \setminus \Sigma$, $x \neq y$,
	\bel{Greenesta} \BAL
	G_{\mu}(x,y)&\asymp |x-y|^{2-N} \left(1 \wedge \frac{d(x)d(y)}{|x-y|^2}\right)  \left(\frac{|x-y|}{d_\Sigma(x)}+1\right)^\am
	\left(\frac{|x-y|}{d_\Sigma(y)}+1\right)^\am \\
	&\asymp |x-y|^{2-N} \left(1 \wedge \frac{d(x)d(y)}{|x-y|^2}\right) \left(1 \wedge \frac{d_\Sigma(x)d_\Sigma(y)}{|x-y|^2} \right)^{-\am}.
	\EAL \ee
	
	(ii) If $\Sigma=\{0\}$ and $\mu = \left( \frac{N-2}{2}\right)^2$ then, for any $x,y \in \Omega \setminus \{0\}$, $x \neq y$,
	\ba\label{Greenestb} \BAL
	&G_{\mu}(x,y) \asymp |x-y|^{2-N} \left(1 \wedge \frac{d(x)d(y)}{|x-y|^2}\right) \left(\frac{|x-y|}{|x|}+1\right)^{\frac{N-2}{2}}
	\left(\frac{|x-y|}{|y|}+1\right)^{\frac{N-2}{2}}\\
	& \qquad \qquad +(|x||y|)^{-\frac{N-2}{2}}\left|\ln\left(1 \wedge \frac{|x-y|^2}{d(x)d(y)}\right)\right| \\
	&\asymp |x-y|^{2-N} \left(1 \wedge \frac{d(x)d(y)}{|x-y|^2}\right) \left(1 \wedge \frac{|x||y|}{|x-y|^2} \right)^{-\frac{N-2}{2}} +(|x||y|)^{-\frac{N-2}{2}}\left|\ln\left(1 \wedge \frac{|x-y|^2}{d(x)d(y)}\right)\right|.
	\EAL \ea
	
	The implicit constants in \eqref{Greenesta} and \eqref{Greenestb} depend on $N,\Omega,\Sigma,\mu$.
\end{proposition}

\begin{proposition}[Sharp two-sided estimates on Martin kernel] \label{Martin} ~~
	
	(i) If $\mu< \left( \frac{N-2}{2}\right)^2$ then
	\be \label{Martinest1}
	K_{\mu}(x,\xi) \asymp\left\{
	\BAL
	&\frac{d(x)d_\Sigma(x)^{-\am}}{|x-\xi|^N}\quad &&\text{if } x \in \Omega \setminus \Sigma,\;  \xi \in \partial\xO, \\
	&\frac{d(x)d_\Sigma(x)^{-\am}}{|x-\xi|^{N-2-2\am}} &&\text{if } x \in \Omega \setminus \Sigma,\; \xi \in \Sigma.
	\EAL \right.
	\ee
	
	(ii) If  $\Sigma=\{0\}$ and $\mu= \left( \frac{N-2}{2}\right)^2$ then
	\ba \label{Martinest1-b}
	K_{\mu}(x,\xi) \asymp\left\{
	\BAL
	&\frac{d(x)|x|^{-\frac{N-2}{2}}}{|x-\xi|^N}\quad &&\text{if } x \in \Omega \setminus \{0\},\; \xi \in \partial\xO, \\
	&d(x)|x|^{-\frac{N-2}{2}}\left|\ln\frac{|x|}{\CD_\Omega}\right| &&\text{if } x \in \Omega \setminus \{0\},\; \xi=0,
	\EAL \right.
	\ea
	where $\CD_\Omega:=2\sup_{x \in \Omega}|x|$.
	
	The implicit constants depend on $N,\Omega,\Sigma,\mu$.
\end{proposition}

\section{Weak Lebesgue estimates of the gradient of Green and Martin operators} \label{sec:weakLp}
We recall that throughout the paper we always assume \eqref{assump1}.
\subsection{Auxiliary estimates} We start this section by recalling the definition of weak Lebesgue spaces (or Marcinkiewicz spaces). Let $D \subset \R^N$ be a domain, $1 < \kappa < \infty$ and $\tau \in \GTM^+(D)$. 
The weak $L^\kappa$ space defined as follows
\bal
L^\kappa_w(D;\tau):=\{ f: D \to \R \text{ Borel measurable}:
\norm{f}_{L^\kappa_w(D;\tau)}<+\infty\}
\eal
where
\bal \norm{f}_{L^\kappa_w(D;\tau)}:=\sup\left\{
\frac{\int_{A}|f|\dd\tau}{\tau(A)^{1-\frac{1}{\kappa}}}: A \sbs D, \, A \text{
	measurable},\, 0<\tau(A)<\infty \right\}. \eal

\noindent When $\dd\tau=\varphi \, \dx$ for some positive continuous function $\varphi$, for simplicity, we use the notation $L_w^\kappa(D;\varphi)$.  Note that the continuous embedding
$ L_w^\kappa(D;\varphi) \sbs L^{r}(D;\varphi)$ holds for any $r \in [1,\kappa)$. We also remark that for any $f \in L_w^\kappa(D;\varphi)$,
\bal \int_{\{x \in D: |f(x)| \geq s\} }\varphi \dd x \leq s^{-\kappa}\norm{f}^\kappa_{L_w^\kappa(D;\varphi)}.
\eal

\noindent Estimates on $\BBG_{\mu}$ and $\BBK_{\mu}$ in $L_w^\kappa$ were obtained in \cite[Theorems 3.8, 3.9, 3.10]{GkiNg_source}.

\begin{theorem}[Weak Lebesgue estimates of $\BBG_{\mu}$] \label{lpweakgreen}
	There holds
	\bel{estgreen}
	\norm{\BBG_\mu[\gt]}_{L_w^{p_*}(\Gw\setminus \Sigma;\ei)} \lesssim \norm{\gt}_{\mathfrak{M}(\xO\setminus \Sigma;\ei)}, \quad \forall \tau\in \mathfrak{M}^+(\xO\setminus \Sigma;\ei).
	\ee
	The implicit constant depends on $N,\Omega,\Sigma,\mu$.
\end{theorem}

%
%

\begin{theorem}[Weak Lebesgue estimates of $\BBK_{\mu}$]\label{lpweakmartin1} ~~
	
\noindent	{\sc I.} Assume $\mu \leq H^2$ and $\gn\in \mathfrak{M}(\partial\xO\cup \Sigma)$ with compact support in $\partial\xO.$ Then
\bal
	\norm{\mathbb{K}_\mu[\nu]}_{L_w^{\frac{N+1}{N-1}}(\Gw\setminus \Sigma;\ei)} \lesssim \|\nu\|_{\mathfrak{M}(\partial\Omega)}.
\eal
The implicit constant depends on $N,\Omega,\Sigma,\mu$.	

\noindent	{\sc II.} Assume $\gn\in \mathfrak{M}(\partial\xO\cup \Sigma)$ with compact support in $\Sigma$.
	
\noindent	(i) If $\mu < \left( \frac{N-2}{2} \right)^2$ then
\bal
	\norm{\mathbb{K}_{\mu}[\nu]}_{L_w^{\frac{N-\am}{N-\am-2}}(\Gw\setminus \Sigma;\ei)} \lesssim \norm{\nu}_{\mathfrak{M}(\Sigma)}.
\eal
The implicit constant depends on $N,\Omega,\Sigma,\mu$.
	
\noindent 	(ii) If  $\Sigma=\{0\}$ and $\mu = \left( \frac{N-2}{2} \right)^2$ then, for any $1<\theta<\frac{N+2}{N-2}$,
\bal
	\norm{\mathbb{K}_{\mu}[\delta_0]}_{L_w^{\theta}(\Gw\setminus \{0\};\ei )} \lesssim 1.
\eal
The implicit constant depends on $N,\Omega,\theta$. Here $\delta_0$ is the Dirac mass concentrated at $0$.
\end{theorem}

In order study problem \eqref{sourceprobrho}, we also need to establish estimates on the gradient of $\BBG_{\mu}$ and $\BBK_{\mu}$ in weak Lebesgue spaces. To this purpose, we invoke a result from \cite{BVi}.

\begin{proposition}[{\cite[Lemma 2.4]{BVi}}] \label{bvivier}
	Assume $D$ is a bounded domain in $\R^N$ and denote by $\tilde D$ either the set $D$ or the boundary $\partial D$. Let $\gw$ be a nonnegative bounded Radon measure in $\tilde D$ and $\eta\in C(D)$ be a positive weight function. Let $\CH$ be a continuous nonnegative function
	on $\{(x,y)\in D\times \tilde D:\;x\neq y\}.$ For any $\xl > 0$ we set
	\bal
	A_\xl(y):=\{x\in D\setminus\{y\}:\;\; \CH(x,y)>\xl\}\quad \text{and} \quad
	m_{\xl}(y):=\int_{A_\xl(y)}\eta(x) \, \dd x.
	\eal
	Suppose that there exist $C>0$ and $\kappa>1$ such that $m_{\xl}(y)\leq C\xl^{-\kappa}$ for every $\gl>0$.  Then the operator
	\bal
{\mathds H}[\gw](x):=\int_{\tilde D}\CH(x,y)\dd\gw(y)
\eal
	belongs to $L^\kappa_w(D;\eta )$ and
	\bal
\left|\left|{\mathds H}[\gw]\right|\right|_{L^\kappa_w(D;\eta)}\leq (1+\frac{C\xk}{\kappa-1})\gw(\tilde D).
\eal
\end{proposition}

Next let us introduce some notations. For $\alpha \in \R$, define
\bel{varphia} \varphi_{\alpha}(x):= d_\Sigma(x)^{-\alpha}d(x), \quad x \in \Omega \setminus \Sigma.
\ee
For $\kappa, \theta,\gamma\in \R$, we define
\bel{F1}
F_{\kappa,\theta}(x,y):=d_\Sigma(x)^{\kappa} |x-y|^{-N+2+\theta}d(y)^{-1}\min(d(x),d_\xS(x),|x-y|)^{-1} \left(  1 \land \frac{d(x)d(y)}{|x-y|^2}\right),
\ee
for $x \neq y, x,y \in \Omega \setminus \Sigma$, and for any positive function $\varphi$ defined on $\Omega \setminus \Sigma$, set
\bal
\mathbb{F}_{\kappa,\theta}[\varphi \tau](x):=\int_{\Omega \setminus \Sigma}F_{\kappa,\theta}(x,y)\varphi(y)\dd\tau(y), \quad \tau \in \GTM(\Omega  \setminus \Sigma;\varphi).
\eal

With the above notations, we obtain the following weak $L^p$ estimate:

\begin{lemma} \label{anisotitaweakF1a}
Let $0<\alpha\leq H$, where $H$ is defined in \eqref{valueH}. Then
\bel{estF1}
	\| \mathbb{F}_{-\xa,2\xa}[\varphi_{\alpha}\tau]\|_{L_w^{\frac{N+1}{N}}(\Gw\setminus \Sigma;\varphi_{\alpha} )} \lesssim
	\|\gt\|_{\mathfrak{M}(\xO\setminus \Sigma;\varphi_{\alpha})}, \quad \forall \tau\in \mathfrak{M}(\xO\setminus \Sigma; \varphi_{\alpha}).
	\ee
	The implicit constant in \eqref{estF1} depends on $N,\Omega,\Sigma,\alpha$.
\end{lemma}

\begin{proof}
Without loss of generality, we may assume that $\tau \geq 0$. Set
\bal
A_\xl(y)&:=\Big\{x\in(\xO\setminus \Sigma)\setminus\{y\}:\;\; F_{-\xa,2\alpha}(x,y)>\xl \Big \},\\ \nonumber
A_{\xl,1}(y)&:=\Big\{x\in(\xO\setminus \Sigma)\setminus\{y\}:\;\; F_{-\xa,2\alpha}(x,y)>\xl\;\text{ and}\;d_\Sigma(x)\leq |x-y| \Big \},\\ \nonumber
A_{\xl,2}(y)&:=\Big\{x\in(\xO\setminus \Sigma)\setminus\{y\}:\;\; F_{-\xa,2\alpha}(x,y)>\xl\;\text{ and}\;d_\Sigma(x)> |x-y| \Big \},\\ \nonumber
m_{\xl}(y)&:=\int_{A_\xl(y)}\varphi_{\alpha} \dx,\quad m_{\xl,i}(y):=\int_{A_{\xl,i}(y)}\varphi_{\alpha}\dx, \quad i=1,2.
\eal
Then
$A_\xl(y)=A_{\xl,1}(y)\cup A_{\xl,2}(y)$
and
\bal
m_{\xl}(y)= m_{\xl,1}(y)+ m_{\xl,2}(y).
\eal

Let $\beta_1$ be as in \eqref{cover}. We write
\ba\label{ml1F2-0}
m_{\xl}(y)=\int_{A_\xl(y)\cap \Sigma_{ \frac{\beta_1}{4} }}d(x)d_\Sigma(x)^{-\alpha} \dx+\int_{A_\xl(y)\setminus \Sigma_{ \frac{\beta_1}{4}}}d(x)d_\Sigma(x)^{-\alpha} \dx.
\ea
We deal only with the case  $H<\frac{N-2}{2}$ since the case $H=\frac{N-2}{2}$ (i.e. $\Sigma=\{0\}$) can be treated in a similar way.

We split the first term on the right hand side of \eqref{ml1F2-0} as
\ba \label{mcompose-1} \int_{A_{\xl}(y)\cap \Sigma_{ \frac{\beta_1}{4} }}d(x)d_\Sigma(x)^{-\alpha} \dx = \int_{A_{\xl,1}(y)\cap \Sigma_{ \frac{\beta_1}{4} }}d(x)d_\Sigma(x)^{-\alpha} \dx + \int_{A_{\xl,2}(y)\cap \Sigma_{ \frac{\beta_1}{4} }}d(x)d_\Sigma(x)^{-\alpha} \dx.
\ea
We note that
\bel{app:6}
1 \land \frac{d(x)d(y)}{|x-y|^2}  \leq2\left( 1 \land \frac{d(x)}{|x-y|} \right)\left( 1 \land \frac{d(y)}{|x-y|} \right)\leq 2\left( 1 \land \frac{d(y)}{|x-y|} \right)  \leq 4  \frac{d(y)}{d(x)}, \; \forall x,y \in \Omega, \; x \neq y,
\ee
therefore
\be\label{52}
F_{-\xa,2\alpha}(x,y)\leq 4 d_\Sigma(x)^{-\xa} d(x)^{-1}\min(d(x),d_\xS(x),|x-y|)^{-1}|x-y|^{-N+2+2\xa},\; \forall x,y \in \Omega \setminus \Sigma, \; x \neq y.
\ee
Since $0<\alpha <\frac{N-2}{2}$, from \eqref{52} we see that
\bal
A_{\xl,1}(y)\cap \Sigma_{ \frac{\beta_1}{4} }
\subset \Big\{x\in(\xO\setminus \Sigma)\setminus\{y\}:\; d_\Sigma(x) < c\lambda^{-\frac{1}{N-1-\alpha}},\;  |x-y| <c\lambda^{-\frac{1}{N-2-2\alpha}}d_\Sigma(x)^{-\frac{\xa+1}{N-2-2\alpha}} \Big \}.
\eal
By using Lemma \ref{lemapp:1} with $\alpha_1=-\alpha$, $\alpha_2=-\frac{\xa+1}{N-2-2\alpha}$, $\ell_1=\lambda^{-\frac{1}{N-1-\alpha}}$, $\ell_2=\lambda^{-\frac{1}{N-2-2\alpha}}$ and noting that $N-k - \alpha -\frac{k(\xa+1)}{N-2-2\alpha} \geq 1$ since $\alpha\leq H$, we deduce, for $\lambda \geq 1$,
\ba \label{ca1-1.1}
\int_{A_{\xl,1}(y)\cap \Sigma_{ \frac{\beta_1}{4} }}d(x)d_\Sigma(x)^{-\alpha} \dx \lesssim \lambda^{-\frac{N-\alpha}{N-1-\alpha}} \leq \lambda^{-\frac{N+1}{N}}.
\ea
Next, by \eqref{52}, we see that
\bal
A_{\lambda,2}(y)\cap \Sigma_{\frac{\beta_1}{4}} \subset \left\{ x \in \Omega \setminus \Sigma:  |x-y| < c\lambda^{-\frac{1}{N-1-\alpha}} \text{ and } d_\Sigma(x) > |x-y| \right\}.
\eal
Therefore, for every $\lambda \geq 1$,
\ba \label{ca1-1.3} \begin{aligned}
\int_{A_{\lambda,2}(y)\cap \Sigma_{\frac{\beta_1}{4}}}d(x)d_\Sigma(x)^{-\alpha}\dx
\lesssim \int_{\{|x-y|\leq c\lambda^{-\frac{1}{N-1-\alpha}}\}}|x-y|^{-\alpha}\dx
\lesssim \lambda^{-\frac{N-\alpha}{N-1-\alpha}} \leq \lambda^{-\frac{N+1}{N}}.
\end{aligned} \ea
Combining \eqref{mcompose-1}, \eqref{ca1-1.1} and \eqref{ca1-1.3} yields, for any $\lambda \geq 1$,
\be \label{AA1}
\int_{A_{\xl}(y)\cap \Sigma_{ \frac{\beta_1}{4} }}d(x)d_\Sigma(x)^{-\alpha} \dx \lesssim \lambda^{-\frac{N+1}{N}}.
\ee

Next we estimate the second term on the right hand side of \eqref{ml1F2-0}. By \eqref{app:6}, we have
\bal
A_\xl(y)\cap (\xO\setminus\Sigma_{ \frac{\beta_1}{4}})\subset\left\{ x \in \Omega \setminus \Sigma:  |x-y| < c\lambda^{-\frac{1}{N}} \text{ and } d(x) \leq \xl^{-1}|x-y|^{-N+1} \right\}.
\eal
This yields, for $\lambda \geq 1$,
\ba \label{AA-2} \BAL
\int_{A_{\lambda}(y) \setminus \Sigma_{\frac{\beta_1}{4}}} d(x)d_{\Sigma}(x)^{-\alpha}\dx &\lesssim\int_{\{|x-y| < c\lambda^{-\frac{1}{N}}\}} \xl^{-1}|x-y|^{-N+1}\dx \lesssim \lambda^{-\frac{N+1}{N}}.
\EAL
\ea

Combining \eqref{ml1F2-0}, \eqref{AA1} and \eqref{AA-2} yields
\bel{ca2-1.5}
m_{\lambda}(y)\leq  C\lambda^{-\frac{N+1}{N}},  \quad \forall \lambda>0,
\ee
where $C=C(N,\Omega,\Sigma,\alpha)$.
By applying Proposition \ref{bvivier} with $\CH(x,y)=F_{-\xa,2\alpha}(x,y),$ $\tilde D=D=\Omega \setminus \Sigma$, $\eta=d\, d_\Sigma^{-\alpha}$ and $\omega=d\, d_\Sigma^{-\alpha} \tau$ and using \eqref{ca2-1.5}, we finally derive \eqref{estF1}.
\end{proof}

By a similar argument as in the proof of Lemma \ref{anisotitaweakF1a}, we can establish the following lemma whose proof is omitted.
\begin{lemma} \label{anisotitaweakF1b}
Let $0<\alpha\leq H$, where $H$ is defined in \eqref{valueH}. Then
\bel{estF1b}
	\| \mathbb{F}_{\xa,0}[\varphi_{\alpha}\tau]\|_{L_w^{\frac{N+1}{N}}(\Gw\setminus \Sigma;\varphi_{\alpha} )} \lesssim
	\|\gt\|_{\mathfrak{M}(\xO\setminus \Sigma;\varphi_{\alpha})}, \quad \forall \tau\in \mathfrak{M}(\xO\setminus \Sigma; \varphi_{\alpha}).
	\ee
	The implicit constant in \eqref{estF1b} depends on $N,\Omega,\Sigma,\alpha$.
\end{lemma}

Let $\varphi_{\frac{N-2}{2}}$ is defined in \eqref{varphia} with $\Sigma=\{0\} \subset \Omega$ and $\alpha=\frac{N-2}{2}$. Set
\bal
\tilde F(x,y):=|x|^{-\frac{N-2}{2}} d(y)^{-1} \left|\ln\left(1 \land \frac{|x-y|^2}{d(x)d(y)}\right)\right|, \quad  x \neq y,\; x,y \in \Omega \setminus \{0\},
\eal
\bal
\mathbb{\tilde F}[\varphi_{\frac{N-2}{2}}\tau](x):=\int_{\Omega \setminus \{0\}}\tilde F(x,y)\varphi_{\frac{N-2}{2}}(y)d\tau(y), \quad \tau \in \mathfrak{M}(\xO\setminus \{0\};\varphi_{\frac{N-2}{2}}).
\eal

\begin{lemma} \label{anisotitaweakF4}
We have
\bal
	\norm{\mathbb{\tilde F}[\varphi_{\frac{N-2}{2}}\gt]}_{L_w^{\frac{N+1}{N}}(\Gw\setminus \{0\};\varphi_{\frac{N-2}{2}})}
	\lesssim \norm{\gt}_{\mathfrak{M}(\xO\setminus \{0\};\varphi_{\frac{N-2}{2}})}, \quad \forall \tau\in \mathfrak{M}(\Omega\setminus \{0\};\varphi_{\frac{N-2}{2}}).
\eal
The implicit constant depends on $N,\Omega$.
\end{lemma}
\begin{proof}
Recall that $\CD_\Omega:=2\sup_{x \in \Omega}|x|$. We note that
\bal
\tilde F(x,y)\leq 2d(y)^{-1}|x|^{-\frac{N-2}{2}}\left(-\ln\frac{|x-y|}{\CD_\Omega}\right) \left(  1 \land \frac{d(x)d(y)}{|x-y|^2} \right), \; \forall x\neq y,\;x,y\in \xO\setminus\{0\}.
\eal
By proceeding as in the proof of Lemma \ref{anisotitaweakF1a}, we obtain the desired result.
\end{proof}

For $\alpha,\xg,\theta \in \R$, put
\ba \label{Halthe} H_{\alpha,\theta,\xg}(x,y):=d(x)^{1-\xg}d_{\Sigma}(x)^{-\alpha-\xg} |x-y|^{-N+\theta}, \quad x \in \Omega \setminus \Sigma, y \in \partial \Omega \cup \Sigma,
\ea
\bal
\mathbb{H}_{\alpha,\theta,\xg}[\nu](x): = \int_{\partial \Omega \cup \Sigma} H_{\alpha,\theta,\xg}(x,y)\dd\nu(y), \quad \nu \in \GTM(\partial \Omega \cup \Sigma).
\eal

\begin{theorem} \label{H}

(i) Assume $k \geq 0$, $\alpha \leq H$, where $H$ is defined in \eqref{valueH}, $0\leq\xg\leq1$  and $\gn\in \mathfrak{M}(\partial\xO\cup \Sigma)$ with compact support on $\partial\xO$. Then
\ba \label{est:H1}
\norm{\mathbb{H}_{\alpha,0,\xg}[\nu]}_{L_w^{ \frac{N+1}{N-1+\xg}}(\Gw\setminus \Sigma;\varphi_{\alpha})} \lesssim \|\nu\|_{\mathfrak{M}(\partial\Omega \cup \Sigma)}.
\ea
The implicit constant in \eqref{est:H1} depends only on $N,\Omega,\Sigma,\alpha,\gamma$.

(ii) Assume $k>0$, $\alpha \leq H<\frac{N-2}{2}$  and $\gn\in \mathfrak{M}(\partial\xO\cup \Sigma)$ with compact support on $\Sigma$. Then
\ba \label{est:H2}
\norm{\mathbb{H}_{\alpha,2(\xa+1),\xg}[\nu]}_{L_w^{\frac{N-\alpha}{N-\alpha+\xg-2}}(\Gw\setminus \Sigma;\varphi_{\alpha})} \lesssim \norm{\nu}_{\mathfrak{M}(\partial \Omega \cup \Sigma)}.
\ea
The implicit constant in \eqref{est:H2} depends only on $N,\Omega,\Sigma,\alpha,\gamma$.

(iii) Assume $\Sigma=\{0\}$ and $\alpha=\frac{N-2}{2}$. For any $0<\xe<\frac{N+2}{N-2+2\xg}-1,$ there holds
\ba \label{est:H3}
\norm{\mathbb{H}_{\frac{N-2}{2},N,\xg}[\delta_0]}_{L_w^{\frac{N+2}{N-2 +2\gamma}-\xe}(\Gw\setminus \{0\};\varphi_{\frac{N-2}{2}})} \lesssim 1.
\ea
The implicit constant in \eqref{est:H3} depends only on $N,\Omega,\gamma,\xe$. Here $\xd_0$ denotes the Dirac measure concentrated at $0$.
\end{theorem}
\begin{proof} For $y\in \partial \Omega \cup \Sigma$, set
\bal
A_\xl(y):=\Big\{x\in(\xO\setminus \Sigma):\; H_{\alpha,\theta,\xg}(x,y)>\xl \Big \}, \quad m_{\xl}(y)&:=\int_{A_\xl(y)}d(x)d_\Sigma(x)^{-\alpha} \dx.
\eal	
We write
\ba \label{split-K}
m_{\xl}(y)=\int_{A_\xl(y)\cap \Sigma_{\xb_1}}d(x)d_\Sigma(x)^{-\alpha} \dx+\int_{A_\xl(y)\setminus \Sigma_{\xb_1}}d(x)d_\Sigma(x)^{-\alpha} \dx.
\ea

(i) Assume $\gn\in \mathfrak{M}(\partial\xO\cup \Sigma)$ with compact support on $\partial\xO$ and without loss of generality, we may assume that $\gn \geq 0$. Let $y \in \partial \Omega$ and $\theta=0.$

First we treat the first term on the right hand side of  \eqref{split-K}. If $-\xg<\alpha \leq H$ then by applying Lemma \ref{lemapp:1}, we obtain, for $\lambda \geq 1$,
\bal
\int_{A_\xl(y)\cap \Sigma_{\xb_1}}d_\Sigma(x)^{-\alpha} \dx \lesssim \int_{\{d_\Sigma(x)\leq c\lambda^{-\frac{1}{\alpha+\xg}}\} \cap \Sigma_{\beta_1}}d_\Sigma(x)^{-\alpha}\dx \lesssim \lambda^{-\frac{N-k-\alpha}{\alpha+\xg}} \leq \lambda^{ - \frac{N+1}{N-1+\xg}}.
\eal
If $\alpha \leq -\xg$ then there exists $\bar C=\bar C(N,\Omega,\Sigma,\alpha,\xg)>1$ such that for any $\lambda>\bar C$, $A_\lambda(y) \cap \Sigma_{\beta_1}=\emptyset$. Consequently, for all $\lambda>\bar C$,
\bel{ml3mb}
\int_{A_\xl(y)\cap \Sigma_{\xb_1}}d_\Sigma(x)^{-\alpha} \dx=0.
\ee

Next we treat the second term on the right hand side of \eqref{split-K}. By using the estimate $d(x) \leq |x-y|$, we see that, for $\lambda \geq 1$,
\ba \label{ml2ma}
\int_{A_\xl(y)\setminus \Sigma_{\xb_1}}d(x)\dx \lesssim \int_{\{|x-y|\leq c\lambda^{-\frac{1}{N-1+\xg}}\}}|x-y| \dx \lesssim \xl^{-\frac{N+1}{N-1+\xg}} .
\ea
Combining \eqref{ml3mb} and \eqref{ml2ma}, we obtain
\ba \label{ml4ma}
m_{\xl}(y)\leq C\lambda^{-\frac{N+1}{N-1+\xg}},
\ea
for all $\lambda>\bar C$, where $C=C(N,\Omega,\Sigma,\alpha,\xg)$. Then we can show that \eqref{ml4ma} holds true for all $\lambda>0$. By applying Proposition \ref{bvivier} with $\mathcal{H}(x,y)=H_{\alpha,0,\xg}(x,y)$, $\tilde D=D=\Omega \setminus \Sigma$, $\eta=\varphi_{\alpha}$ and $\omega=\xn$, we obtain \eqref{est:H1}.

(ii) Assume $\gn\in \mathfrak{M}(\partial\xO\cup \Sigma)$ with compact support on $\Sigma$ and without loss of generality, we may assume that $\gn \geq 0$. Let $y \in \Sigma$ and $\theta=2(\xa+1)$. \medskip

\noindent \textbf{Case 1:} $-\xg<\alpha \leq H$.
First we treat the first term in \eqref{split-K}. We notice that since $y \in \Sigma$, $d_\Sigma(x) \leq |x-y|$ for every $x \in \Omega \setminus \Sigma$, hence
\bal
A_\lambda(y) \subset \{ x \in \Omega \setminus \Sigma: d_\Sigma(x) \leq c\lambda^{-\frac{1}{N-\alpha+\xg-2}} \quad \text{and} \quad |x-y|<c\lambda^{-\frac{1}{N-2\xa-2}}d_\Sigma(x)^{-\frac{\alpha+\xg}{N-2\xa-2}}   \}.
\eal
Therefore, by applying Lemma \ref{lemapp:1}, we obtain
\ba \label{ml3maK-1}
\int_{A_\xl(y)\cap \Sigma_{\beta_1}}d_\Sigma(x)^{-\alpha} \dx \lesssim \lambda^{-\frac{N-\alpha}{N-\alpha+\xg-2}}.
\ea

Next we treat the second term in \eqref{split-K}. We see that there is $\bar C=\bar C(N,\Omega,\Sigma,\alpha,\theta)>1$ such that for any $\lambda>\bar C$, $\int_{A_\xl(y)\setminus \Sigma_{\beta_1}}d(x)\dx=0$. This, together with
\eqref{split-K}, \eqref{ml3maK-1}, implies
\ba \label{mlam-1}
m_{\xl}(y)\leq C\,\lambda^{-\frac{N-\alpha}{N-\alpha+\xg-2}}
\ea
for all $\lambda > \hat C$, where $C=C(N,\Omega,\Sigma,\alpha,\theta)$. \medskip

\noindent \textbf{Case 2:} $\alpha \leq -\xg$. By noting that $d_\Sigma(x)^{-\alpha-\xg} \leq |x-y|^{-\alpha-\xg}$ and $|x-y| \leq c \lambda^{-\frac{1}{N-2-\xa+\xg}}$ for every $x \in A_\lambda(y)$, we can easily obtain \eqref{mlam-1}.

From case 1 and case 2, by applying Proposition \ref{bvivier} with $\mathcal{H}(x,y)=H_{\alpha,2(\xa+1),\xg}(x,y)$, $D = \Omega \setminus \Sigma$, $\tilde D=\partial\xO\cup \xS,$ $\eta=\varphi_{\alpha}$ and $\omega=\xn$, we obtain \eqref{est:H2}. 

\medskip

(iii) The proof is very similar to the one in statement (ii) and is omitted.
\end{proof}
\smallskip

\subsection{Estimate on the gradient of Green kernel and Martin kernel}
In this subsection, we will establish estimates on the gradient of the Green kernel. A standard scaling argument implies that there exists a positive constant $C$ depending only on $N,\xO,\xS,\xm$ such that
\ba\label{gradientgreen}
|\nabla_x G_{\mu}(x,y)|\leq C\frac{G_\xm(x,y)}{\min(d(x),d_\xS(x),|x-y|)},\quad\forall x\neq y\;\;\text{ and } x,y\in\xO\setminus\xS.
\ea

Let $\ei$ be the eigenfunction of $-L_\mu$ (see Appendix \ref{subsect:eigen}), $\varphi_{\alpha}$ be as in \eqref{varphia} and $\BBG_{\mu}$ be as in \eqref{BBG}. For a measure $\tau$ on $\Omega \setminus \Sigma$, we have
\bal
\nabla\BBG_\mu[\tau](x)=\int_{\xO \setminus \Sigma}\nabla_x G_{\mu}(x,y)\dd\tau(y).
\eal
A combination of estimates in the previous subsection and estimate \eqref{gradientgreen} leads to estimates on the gradient of the Green operator in weak Lebesgue spaces.
\begin{theorem} \label{lpweakgrangreen}
Assume $k \geq 0$ and $0<\mu \leq H^2$.
Then
\ba \label{estgrangreen}
	\norm{\nabla\BBG_\mu[\gt]}_{L_w^{\frac{N+1}{N}}(\Gw\setminus \Sigma;\ei)} \lesssim \norm{\gt}_{\mathfrak{M}(\xO\setminus \Sigma;\ei)}, \quad \forall \tau\in \mathfrak{M}(\xO\setminus \Sigma;\ei).
\ea
	The implicit constant depends on $N,\Omega,\Sigma,\mu$.
\end{theorem}

\begin{proof} Without loss of generality we may assume that $\tau$ is nonnegative. 
	
\noindent \textbf{Case 1: $0<\mu<\left( \frac{N-2}{2} \right)^2$.} Then $0<\am< \frac{N-2}{2}$.  From \eqref{eigenfunctionestimates}, \eqref{Greenesta}, \eqref{F1} and the fact that $d_\Sigma(y) \leq |x-y|+d_\Sigma(x)$, we  obtain, for all $x,y\in \xO\setminus \Sigma, x\neq y$,
\bal
G_\mu(x,y)\varphi_{\am,\xg}(y)^{-1} &\lesssim |x-y|^{2-N} \min \left\{ 1, \frac{d(x)d(y)}{|x-y|^2} \right\} (|x-y|+d_\Sigma(x))^{2\am} d_\Sigma(x)^{-\am}d(y)^{-\xg} \\
&\lesssim (F_{-\am,2\am}(x,y)+F_{\am,0}(x,y))\min(d(x),d_\xS(x),|x-y|).
\eal
This, together with \eqref{gradientgreen}, Lemmas \ref{anisotitaweakF1a}--\ref{anisotitaweakF1b} and estimate  $\varphi_{\am} \asymp \phi_\mu$,
implies \eqref{estgrangreen}. \medskip

\noindent \textbf{Case 2:} $\Sigma=\{0\}$ and $\mu=\left( \frac{N-2}{2} \right)^2$. In this case $\am=\frac{N-2}{2}$.
From \eqref{eigenfunctionestimates}, \eqref{Greenestb} and the fact that $|y| \leq |x-y|+|x|$, we obtain, for all  $x,y\in \xO\setminus \{0\}, x\neq y$,
\bal
G_{(\frac{N-2}{2})^2}(x,y)\varphi_{\frac{N-2}{2}}(y)^{-1}\lesssim (F_{-\frac{N-2}{2},N-2}(x,y)+F_{\frac{N-2}{2},0}(x,y)+\tilde F(x,y))\min(d(x),|x|,|x-y|).
\eal
This, together with \eqref{gradientgreen}
and Lemmas \ref{anisotitaweakF1a}--\ref{anisotitaweakF4},
implies \eqref{estgrangreen}. The proof is complete.
\end{proof}

Next we treat the case $\mu \leq 0$. Recall that $q_*$ is defined in \eqref{p*q*}.
\begin{theorem} \label{lpweakgrangreen2}
Assume $\xm\leq 0$. Then
\ba \label{estgangreen2}
	\norm{\nabla\BBG_\mu[\gt]}_{L_w^{q_*}(\Gw\setminus \Sigma;\ei)} \lesssim \norm{\gt}_{\mathfrak{M}(\xO\setminus \Sigma;\ei)}, \quad \forall \tau\in \mathfrak{M}^+(\xO\setminus \Sigma;\ei).
\ea
	The implicit constant depends on $N,\Omega,\Sigma,\mu$.
\end{theorem}
\begin{proof}
\noindent For $y \in \Omega \setminus \Sigma$ and $\lambda>0$, set
\bal
A_\xl(y)&:=\Big\{x\in(\xO\setminus \Sigma)\setminus\{y\}:\;\; |\nabla_xG_{\mu}(x,y)|\ei(y)^{-1}>\xl \Big \}, \quad
m_{\xl}(y):=\int_{A_\xl(y)}d(x)d_\Sigma(x)^{-\am} \dd x.
\eal
Put
\bal
 F(x,y)&:=d_\Sigma(y)^{\am}|x-y|^{-N+2} d(y)^{-1}\min(d(x),d_\xS(x),|x-y|)^{-1} \\
 &\qquad \times \left( 1 \land \frac{d(x)d(y)}{|x-y|^2} \right) \left(1 \land \frac{d_\Sigma(x)d_\Sigma(y)}{|x-y|^2} \right)^{-\am},
  \;\; x,y \in \Omega \setminus \Sigma, x \neq y.
\eal
By \eqref{Greenesta}, \eqref{eigenfunctionestimates} and \eqref{gradientgreen}, $F(x,y) \geq c\,|\nabla_xG_{\mu}(x,y)|\varphi_{\xa}(y)^{-1}$ for  $c>0$ depending only on $N,\xO,\Sigma,\mu.$
Consequently,
$
A_\lambda(y)\subset \Big\{x\in(\xO\setminus \Sigma)\setminus\{y\}:\; F(x,y)>c \lambda  \Big \}=:\tilde A_\lambda(y).
$
Let $\beta_0$ be as in Subsection \ref{assumptionK}. We see that
\ba
m_{\xl}(y)
\lesssim \int_{\tilde A_\xl(y)\cap \Sigma_{\xb_0}}d_\Sigma(x)^{-\am} \dd x+\int_{\tilde A_\xl(y)\setminus \Sigma_{\xb_0}}d(x)\dd x.
\label{ml1gr}
\ea
Note that, for $\Gamma=\partial \xO$ or $\xS$, we have
\bal
  1 \land \frac{d_\Gamma(x)d_\Gamma(y)}{|x-y|^2}\leq2 \left( 1 \land \frac{d_\Gamma(x)}{|x-y|} \right)\left( 1 \land \frac{d_\Gamma(y)}{|x-y|} \right) \leq 2 \left( 1 \land \frac{d_\Gamma(y)}{|x-y|} \right) \leq 4\frac{d_\Gamma(y)}{d_\Gamma(x)}.
\eal
Therefore
\ba
\int_{\tilde A_\xl(y)\setminus \Sigma_{\xb_0}}d(x) \dd x \lesssim \int_{\{|x-y|\leq c\lambda^{-\frac{1}{N}}\}}\xl^{-1}|x-y|^{-N+1} \dd x \lesssim \lambda^{-\frac{N+1}{N}}. \label{ml2gr}
\ea

\noindent \textbf{Case 1:} $\am\leq-1.$ Then
\ba
\int_{\tilde A_\xl(y)\cap \Sigma_{\xb_0}}d_\Sigma(x)^{-\am} \dd x \lesssim \int_{\{|x-y|\leq c \lambda^{-\frac{1}{N-\am-1}}\}}\xl^{-1}|x-y|^{-N+1} \dd x \lesssim \lambda^{-\frac{N-\am}{N-\am-1}}.\label{ml3gr}
\ea
Combining \eqref{ml1gr}, \eqref{ml2gr} and \eqref{ml3gr}, we obtain
\bel{ml4gr}
 m_{\xl}(y)\leq C \lambda^{-p_*}
\ee
for all $\lambda \geq 1$, where $C=C(N,\Omega,\Sigma,\mu)$. Then we can show that \eqref{ml4gr} holds for every $\lambda>0$.

\noindent \textbf{Case 2:} $-1<\am\leq0.$ Proceeding as in the proof of Lemma \ref{anisotitaweakF1a}, we may also deduce that
$m_{\xl}(y)\leq C \lambda^{-p_*}$ for all $\lambda>0$. 

From the above two cases, by applying Proposition \ref{bvivier}  with $\mathcal{H}(x,y)=|\nabla_xG_{\mu}(x,y)|\ei(y)^{-1}, $  $\tilde D=D=\xO\setminus \Sigma$, $\eta=\ei$ and $\omega=\ei\tau$, we obtain \eqref{estgangreen2}. The proof is complete.
\end{proof}

Next we derive estimates of the gradient of the Martin operator in weak Lebesgue spaces. By standard elliptic regularity results, one can  check that for any $\nu \in \GTM(\Omega \cup \Sigma)$, 
\ba \label{est:gradK-K}
|\nabla \BBK_\mu[\nu](x)|\lesssim d(x)^{-1}d_\xS(x)^{-1}|\BBK_\mu[\nu](x)| \quad \forall x \in \Omega \setminus \Sigma.
\ea
Note that
\bal
\nabla \mathbb{K}_\mu[\gn](x)=\int_{\partial\xO \cup \Sigma}\nabla_x K_{\mu}(x,y) \dd\xn(y), \quad \gn\in \mathfrak{M}(\partial\xO\cup \Sigma).
\eal

\begin{theorem}\label{lpweakgranmartin1} {\sc I.} Assume $\mu \leq \left( \frac{N-2}{2}\right)^2,$ $0\leq\xg\leq1$ and $\gn\in \mathfrak{M}(\partial\xO\cup \Sigma)$ with compact support on $\partial\xO.$ Then
	\ba \label{estgranmartin1}
	\norm{(dd_\xS)^{-\xg}\mathbb{K}_\mu[\nu]}_{L_w^{\frac{N+1}{N-1+\xg}}(\Gw\setminus \Sigma;\ei)} \lesssim \|\nu\|_{\mathfrak{M}(\partial\Omega \cup \Sigma)}.
	\ea
	The implicit constant in the above estimate depends on $N,\Omega,\Sigma,\mu,\gamma$.

	{\sc II.} Assume $\gn\in \mathfrak{M}(\partial\xO\cup \Sigma)$ with compact support on $\Sigma$.
	
	(i) If $\mu < \left( \frac{N-2}{2} \right)^2$ then
	\ba \label{estgranmartin2}
	\norm{ (dd_\xS)^{-\xg}\mathbb{K}_{\mu}[\nu]}_{L_w^{\frac{N-\am}{N-\am-2+\xg}}(\Gw\setminus \Sigma;\ei)} \lesssim \norm{\nu}_{\mathfrak{M}(\partial \Omega \cup \Sigma)}.
	\ea
	The implicit constant in the above estimate depends on $N,\Omega,\Sigma,\mu,\gamma$.
	
	(ii) Assume $\Sigma=\{0\}$ and $\mu = \left( \frac{N-2}{2} \right)^2$. Then for any $1<q<\frac{N+2}{N-2+2\xg}$,
	\ba \label{estgranmartin2cr}
	\norm{(d(\cdot)|\cdot|)^{-\xg} \mathbb{K}_{\mu}[\delta_0]}_{L_w^{q}(\Gw\setminus \{0\};\ei )} \lesssim 1.
	\ea
	The implicit constant in the above estimate depends on $N,\Omega,\gamma,q$.
\end{theorem}
\begin{proof} In view of estimate \eqref{est:gradK-K}, it is enough to establish upper bounds for $(dd_\xS)^{-\xg}|\BBK_\mu[\nu]$.
	
	I. By applying Theorem \ref{H} (i) with $\alpha=-\am$ and noting that $(d(x)d_{\Sigma}(x))^{-\gamma}K_\mu(x,y) \asymp H_{\am,0,\gamma}(x,y)$ for  $(x,y) \in (\xO\setminus\xS)\times\partial\xO$ (due to \eqref{Martinest1} and \eqref{Halthe}),  $\varphi_{\am} \asymp \phi_\mu$ (due to\eqref{eigenfunctionestimates}), we obtain 
	\bal 
	\norm{(dd_\xS)^{-\xg}\mathbb{K}_\mu[\nu]}_{L_w^{\frac{N+1}{N-1+\xg}}(\Gw\setminus \Sigma;\ei)} \lesssim \|\nu\|_{\mathfrak{M}(\partial\Omega \cup \Sigma)}.
	\eal
	This and estimate \eqref{est:gradK-K} imply \eqref{estgranmartin1}.
	
	II (i). By applying Theorem \ref{H} (ii) with $\alpha=\am$ and noting that $(d(x)d_{\Sigma}(x))^{-\gamma}K_\mu(x,y) \asymp H_{\am,2(1+\am),\gamma}(x,y)$ for  $(x,y) \in (\xO\setminus\xS)\times\xS$ (due to \eqref{Martinest1} and \eqref{Halthe}),  $\varphi_{\am} \asymp \phi_\mu$ (due to\eqref{eigenfunctionestimates}), we obtain 
	\bal 
	\norm{(dd_\xS)^{-\xg}\mathbb{K}_{\mu}[\nu]}_{L_w^{\frac{N-\am}{N-\am-2+\xg}}(\Gw\setminus \Sigma;\ei)} \lesssim \norm{\nu}_{\mathfrak{M}(\partial \Omega \cup \Sigma)}.
	\eal
	This and estimate \eqref{est:gradK-K} imply \eqref{estgranmartin2}.
	
	II (ii). By applying Theorem \ref{H} (iii), we obtain 
	\bal 
	\norm{(d(\cdot)|\cdot|)^{-\xg}\mathbb{K}_{\mu}[\nu]}_{L_w^{q}(\Gw\setminus \{0\};\ei )} \lesssim \norm{\nu}_{\mathfrak{M}(\partial \Omega \cup \{0\})}.
	\eal
	This and estimate \eqref{est:gradK-K} imply \eqref{estgranmartin2cr}.
\end{proof}

As a consequence of \eqref{est:gradK-K} and Theorem \ref{lpweakgranmartin1}, we  obtain bounds on the gradient of $\BBK_{\mu}$.

\begin{theorem}\label{lpweakgranmartin1-0} {\sc I.} Assume $\mu \leq \left( \frac{N-2}{2}\right)^2$ and $\gn\in \mathfrak{M}(\partial\xO\cup \Sigma)$ with compact support on $\partial\xO.$ Then
\bal 
\norm{\nabla \mathbb{K}_\mu[\nu]}_{L_w^{\frac{N+1}{N}}(\Gw\setminus \Sigma;\ei)} \lesssim 	\norm{(dd_\xS)^{-1}\mathbb{K}_\mu[\nu]}_{L_w^{\frac{N+1}{N}}(\Gw\setminus \Sigma;\ei)} \lesssim \|\nu\|_{\mathfrak{M}(\partial\Omega \cup \Sigma)}.
\eal
The implicit constant depends only on $N,\Omega,\Sigma,\mu$.

{\sc II.} Assume $\gn\in \mathfrak{M}(\partial\xO\cup \Sigma)$ with compact support on $\Sigma$.

(i) If $\mu < \left( \frac{N-2}{2} \right)^2$ then
\bal 
\norm{\nabla \mathbb{K}_{\mu}[\nu]}_{L_w^{\frac{N-\am}{N-\am-1}}(\Gw\setminus \Sigma;\ei)}  \lesssim \norm{ (dd_\xS)^{-1}\mathbb{K}_{\mu}[\nu]}_{L_w^{\frac{N-\am}{N-\am-1}}(\Gw\setminus \Sigma;\ei)}\lesssim \norm{\nu}_{\mathfrak{M}(\partial \Omega \cup \Sigma)}.
\eal
The implicit constant in the above estimates depends on $N,\Omega,\Sigma,\mu$.

(ii) Assume $\Sigma=\{0\}$ and $\mu = \left( \frac{N-2}{2} \right)^2$. Then for any $1<q<\frac{N+2}{N}$,
\bal 
\norm{\nabla  \mathbb{K}_{\mu}[\delta_0]}_{L_w^{q}(\Gw\setminus \{0\};\ei )} \lesssim \norm{(d(\cdot) |\cdot|)^{-1} \mathbb{K}_{\mu}[\delta_0]}_{L_w^{q}(\Gw\setminus \{0\};\ei )} \lesssim 1.
\eal
The implicit constant in the above estimate depends on $N,\Omega,q$.
\end{theorem}

\section{Nonlinear equations with subcritical source} \label{subcritical source}
We recall that throughout the paper we always assume \eqref{assump1}.

\subsection{Subcritical source}
In this section we study the problem \eqref{sourceprobrho}. 
We start with a uniform integrability result.

\begin{lemma} \label{subcrcongran} Assume that $g: {\mathbb R} \times {\mathbb R}_+ \to {\mathbb R}_+$ is nondecreasing and locally Lipschitz in its two variables with $g(0,0)=0$ such that
	\ba \label{subcd0gran} \int_1^\infty  s^{-p-1}g(s,s^{\frac{p}{q}}) \dd s<+\infty
	\ea
	for $q>0$ and $p >0$. Let $u,v$ be measurable functions in $\Omega \setminus \Sigma$. For $s>0$, set
	\bal E_v(s):=\{x\in \xO\setminus \Sigma:| v(x)|>s\} \quad \text{and} \quad e_v(s):=\int_{E_s(v)} \ei \dd x.
	\eal
	Assume that there exist positive constants $C_u,C_v$ such that
	\bal
	e_u(s) \leq C_us^{-p}\quad\text{and}\quad e_v(s) \leq C_vs^{-q}\quad \forall s>0.
	\eal
	Then for any $s_0>0,$ there holds
	\bal
	\norm{g(u,v)}_{L^1(\Omega;\ei)} \leq \int_{E_u(s_0)^c\cap E_v(s_0^\frac{p}{q})^c} g(u,v)\ei\,\dd x
	+2p(C_u+C_v) \int_{s_0}^{\infty}g(s,s^\frac{p}{q})s^{-1-p}\, \dd s.
	\eal
\end{lemma}
\begin{proof} For $s>0$, set
	\bal
	e_{u,v}(s):=\int_{E_s(u)\cap E_v(s^\frac{p}{q})} \ei \dd x.
	\eal
	Then
	\ba \nonumber 
	\int_\xO g(u,v)\ei\,\dx&= \int_{E_u(s_0)^c\cap E_v(s_0^\frac{p}{q})^c} g(u,v)\ei\,\dd x+\int_{E_u(s_0)\cap E_{v}(s_0^\frac{p}{q})^c} g(u,v)\ei\,\dd x\\ \label{g(u,v)}
	&\quad\quad+ \int_{E_u(s_0)^c \cap E_{v}(s_0^\frac{p}{q})} g(u,v)\ei\,\dd x+\int_{E_u(s_0)\cap E_{v}(s_0^\frac{p}{q})} g(u,v)\ei\,\dd x.
	\ea
	We have
	\bal\BAL
	\int_{E_u(s_0)\cap E_{v}(s_0^\frac{p}{q})^c} g(u,v)\ei\,\dd x &\leq -\int_{s_0}^\infty g(s,s_0^\frac{p}{q})\, \dd e_{u}(s)\leq p C_u\int_{s_0}^{\infty}g(s,s^\frac{p}{q}_0)s^{-1-p}\, \dd s\\
	&\leq p C_u\int_{s_0}^{\infty}g(s,s^\frac{p}{q})s^{-1-p}\, \dd s,
	\EAL
	\eal
	\bal
	\BAL
	\int_{E_u(s_0)^c \cap E_{v}(s_0^\frac{p}{q})} g(u,v)\ei\,dx\leq -\int_{s_0}^\infty g(s_0,s^\frac{p}{q})\, \dd e_{v}(s^{\frac{p}{q}})\leq p C_v\int_{s_0}^{\infty}g(s,s^\frac{p}{q})s^{-1-p}\, \dd s,
	\EAL
	\eal
	\bal\BAL
	\int_{E_v(s_0)\cap E_{u}(s_0^\frac{p}{q})} g(u,v)\ei\,dx\leq -\int_{s_0}^\infty g(s,s^\frac{p}{q})\, \dd e_{u,v}(s)\leq p\min(C_u,C_v)\int_{s_0}^{\infty}g(s,s^\frac{p}{q})s^{-1-p}\, \dd s.
	\EAL
	\eal
	By plugging the above estimates into \eqref{g(u,v)}, we obtain the desired result.
\end{proof}

We first establish an existence result for the case where $g$ is smooth and bounded.

\begin{lemma} \label{transf} Let $\nu \in \GTM^+(\prt \Gw)$ with $\norm{\nu}_{\GTM(\prt \Gw)}=1$ and $g \in C^1(\BBR \times \BBR_+) \cap L^\infty(\BBR \times \BBR_+)$. Assume \eqref{integralcond} and \eqref{multi} are satisfied. Then there exists  $\varrho_0>0$ depending on $N,\mu,\Gw,\Gl_g,\tl k$ such that for every $\varrho \in (0,\varrho_0)$ the following problem
	\bel{trans} \left\{ \BAL -L_\mu v &= g (v + \varrho\BBK_\mu[\nu],|\nabla(v + \varrho\BBK_\mu[\nu])|) \quad \text{in } \Gw\setminus\xS, \\ \tr(v) &= 0
	\EAL \right. \ee
	admits a positive weak solution $v$ satisfying
	\bel{estM} \|v\|_{L_w^{p_*}(\Gw\setminus\xS;\ei)} + \|\nabla v\|_{L_w^{q_*}(\Gw\setminus\xS;\ei)} \leq t_0
	\ee
	where $t_0>0$ depending on $N,\mu,\Gw,\Gl_g,\tl k,\tl p,\tl q$. Here $\Gl_g$ is defined in \eqref{integralcond} and $\tl k,\tl p,\tl q$ are the constants in \eqref{multi}.
\end{lemma}
\begin{proof} We shall use the Schauder fixed point theorem to show the existence of a positive weak solution of \eqref{trans}. Define the operator $\BBS$ by
	\bal\BBS(v):=\BBG_\mu[g(v + \varrho\BBK_\mu[\nu], |\nabla(v + \varrho \BBK_\mu[\nu])|) ], \quad v \in W^{1,1}(\Gw\setminus\xS;\ei). 
	\eal
	Fix $1<\kappa<\min\{\tilde p,\tilde q, q_*\}$ and
	\bal \BAL Q_1(v) &: = \| v\|_{L^{p_*}_w(\Gw\setminus\xS;\ei)} \quad &&\text{for } v \in L^{p_*}_w(\Gw\setminus\xS;\ei), \\
	Q_2(v) &:=\|\nabla v\|_{L_w^{q_*}(\Gw\setminus\xS;\ei)} \quad &&\text{for } |\nabla v| \in L_w^{q_*}(\Gw\setminus\xS;\ei),\\
	Q_3(v) &: = \| v\|_{L^{\kappa}(\Gw;\ei)} \quad &&\text{for } v \in L^{\kappa}(\Gw;\ei),\\
	Q_4(v) &:=\|\nabla v\|_{L^{\kappa}(\Gw;\ei)} \quad &&\text{for } |\nabla v| \in L^{\kappa}(\Gw;\ei),\\
	Q(v)&:=Q_1(v) + Q_2(v)+ Q_3(v)+ Q_4(v).
	\EAL \eal
	
	\noindent \textbf{Step 1:} To estimate $L^1(\Gw;\ei)$-norm of  $g(v + \varrho\BBK_\mu[\nu],|\nabla(v+\varrho\BBK_\mu[\nu])|)$.
By Lemma \ref{subcrcongran} with $u=|v+\varrho \BBK_\mu[\nu]|$, $v=|\nabla(v+\varrho\BBK_\mu[\nu])|$ and Theorem \ref{lpweakgranmartin1}, we	can show that there exists a positive constant $C=C(N,\mu,\Omega,\tl k,\Lambda_g)$ such that
	\bel{lk5} \BAL \|g(v+\varrho \BBK_\mu[\nu], |\nabla(v+\varrho\BBK_\mu[\nu])|)\|_{L^1(\Gw\setminus\xS;\ei)}
	\leq C\left(Q_1(v)^{p_*}+Q_2(v)^{q_*}+Q_3(v)^\kappa+Q_4(v)^\kappa + \varrho^\kappa \right).
	\EAL \ee
	
	\noindent \textbf{Step 2:} To estimate $Q_i$, $i=1,2,3,4$, and $Q$.
	
	By \eqref{estgreen}, we have
	\bal \BAL Q_1(\BBS(v)) \leq c \norm{g(v + \varrho \BBK_\mu[\nu], |\nabla(v+\varrho\BBK_\mu[\nu])|)}_{L^1(\Gw;\ei)}.
	\EAL \eal
	This and \eqref{lk5} implies that
	\bal Q_1(\BBS(v)) \leq C\left(Q_1(v)^{p_*}+Q_2(v)^{q_*}+Q_3(v)^\kappa +Q_4(v)^\kappa + \varrho^\kappa \right)
	\eal
	where  $C=C(N,\mu,\Omega,\tl k,\Lambda_g)$. Next we deduce from \eqref{estgrangreen} that
	\bal
\BAL Q_2(\BBS(v)) &= \|\nabla\BBG_\mu[g(v + \varrho \BBK_\mu[\nu], |\nabla(v + \varrho \BBK_\mu[\nu])|)]\|_{L^{q_*}_w(\Gw\setminus\xS;\ei)} \\
	&\leq c \|g(v + \varrho \BBK_\mu[\nu],|\nabla(v+\varrho\BBK_\mu[\nu])|)\|_{L^1(\Gw;\ei)},
	\EAL
\eal
	which in turn implies
	\bal Q_2(\BBS(v)) \leq C\left(Q_1(v)^{p_*}+Q_2(v)^{q_*}+Q_3(v)^\kappa +Q_4(v)^\kappa + \varrho^\kappa \right)
	\eal
	where $C=C(N,\mu,\Omega,\tl k,\Lambda_g)$.
	By \eqref{estgreen} and \eqref{estgrangreen}, we can easily deduce that
	\bal
	\BAL Q_3(\BBS(v)) &= \| \BBG_\mu[g(v + \varrho \BBK_\mu[\nu],|\nabla(v + \varrho \BBK_\mu[\nu])|)]\|_{L^{\kappa}(\Gw;\ei)} \\[2mm]
	&\leq c \norm{g(v + \varrho \BBK_\mu[\nu], |\nabla(v+\varrho\BBK_\mu[\nu])|)}_{L^1(\Gw;\ei)},
	\EAL \eal
	and
	\bal \BAL Q_4(\BBS(v)) &= \|\nabla\BBG_\mu[g(v + \varrho \BBK_\mu[\nu], |\nabla(v + \varrho \BBK_\mu[\nu])|)]\|_{L^{\kappa}(\Gw;\ei)} \\
	&\leq c \|g(v + \varrho \BBK_\mu[\nu],|\nabla(v+\varrho\BBK_\mu[\nu])|)\|_{L^1(\Gw;\ei)}.
	\EAL \eal
	Thus,
	\bal Q_3(\BBS(v))+Q_4(\BBS(v)) \leq C \left(Q_1(v)^{p_\xm}+Q_2(v)^{q_*}+Q_3(v)^\kappa+Q_4(v)^\kappa + \varrho^\kappa \right)
	\eal
	where $C(N,\mu,\Omega,\tl k,\Lambda_g)$.
	Consequently,
	\bal Q(\BBS(v))\leq C\left(Q_1(v)^{p_*}+Q_2(v)^{q_*}+Q_3(v)^\kappa +Q_4(v)^\kappa + \varrho^\kappa \right).
	\eal
	
	Therefore if $Q(v) \leq t$ then
	$Q(\BBS(v)) \leq C\left(t^{p_*}+t^{q_*}+2t^\kappa +\varrho^\kappa \right)$.
	Since $p_*>q_*>\kappa>1$, there exists $\varrho_0>0$ depending on $N,\mu,\Gw,\tl k,\Gl_g$ such that for any $\varrho \in (0,\varrho_0)$ the equation
	$C\left(t^{p_*}+t^{q_*}+2t^\kappa +\varrho^\kappa \right) = t$ 
	admits a largest root $t_0>0$ which depends on $N,\mu,\Gw,\Gl_g, \tl k$. Therefore,
	\bel{ul11} Q(v) \leq t_0  \Longrightarrow Q(\BBS(v)) \leq t_0. \ee
	
	\noindent \textbf{Step 3:} We apply Schauder fixed point theorem to our setting.

	\textit{We claim that $\BBS$ is continuous}. Indeed, if $u_n\rightarrow u$ in $W^{1,1}(\Gw\setminus\xS;\ei)$ then since $g \in C^1(\BBR\times\BBR_+ ) \cap L^\infty(\BBR\times\BBR_+)$, it follows that $g(u_n + \varrho \BBK_\mu[\nu],|\nabla(u_n + \varrho \BBK_\mu[\nu])|) \to g(u + \varrho \BBK_\mu[\nu],|\nabla(u + \varrho \BBK_\mu[\nu])|)$ in $L^{1}(\Gw;\ei)$. Hence, $\BBS(u_n)\to \BBS(u)$ in $W^{1,1}(\Gw;\ei)$.
	
	\textit{Next we claim that $\BBS$ is compact}. Indeed, set
	\bal M:=\sup_{t,s>0}|g(t,s)|<+\infty. \eal
	Then we can easily deduce that there exists $C=C(N,\xO,\Sigma,M,\mu)$ such that
	\be\label{sup1}
	|\BBS(w)(x)|\leq Cd(x)d_{\Sigma}(x)^{-\max(\am,0)},\quad \forall x\in \xO\setminus\xS \; \text{ and } \;\forall w\in W^{1,1}(\Gw;\ei),
	\ee
	\be\label{sup2}
	|\nabla \BBS(w)(x)|\leq Cd_\xS(x)^{-\max(\am,0)-1}, \quad \forall x\in \xO\setminus\xS \; \text{ and } \;\forall w\in W^{1,1}(\Gw;\ei).
	\ee
	Let $\{u_n\}$ be a sequence in $W^{1,1}(\Gw\setminus\xS;\ei)$ then by \eqref{sup1} and \eqref{sup2}, $\{\BBS(u_n)\}$ is uniformly bounded in $W^{1,1}(\Gw\setminus\xS;\ei)$. Therefore there exists $\psi\in W^{2,p}_{\mathrm{loc}}(\xO\setminus\xS) $ and a subsequence still denoted by $\{\BBS(u_n)\}$ such that $\BBS(u_n)\rightarrow \psi$ in $L^p_{\mathrm{loc}}(\xO)$ and
	$\nabla \BBS(u_n)\to \nabla \psi $ in $L^p_{\mathrm{loc}}(\xO\setminus\xS)$ and a.e. in $\xO$. By dominated convergence theorem we have that $\BBS(u_n)\rightarrow \psi$ in $W^{1,1}(\Gw\setminus\xS;\ei)$.
	
	Now set
	\be
	\CO:=\{ \xi \in W^{1,1}(\Gw\setminus\xS;\ei): Q(u) \leq t_0  \}.\label{O}
	\ee
	Then $\CO$ is a closed, convex subset of $W^{1,1}(\Gw\setminus\xS;\ei)$ and by \eqref{ul11}, $\BBS(\CO) \sbs \CO$.
	Thus we can apply the Schauder fixed point theorem to obtain the existence of a function $v \in \CO$ such that $\BBS(v)=v$. This means that $v$ is a nonnegative solution of \eqref{trans} and hence there holds
	\bal  -\int_{\Gw}v L_\gm\zeta dx= \int_{\Gw}  g(v + \varrho \BBK_\mu[\nu],|\nabla(v + \varrho \BBK_\mu[\nu])|) \zeta dx  \forevery \zeta \in {\bf X}_\mu(\Gw\setminus\xS). \eal
The proof is complete.
\end{proof} \medskip

\begin{proof}[\textbf{Proof of Theorem \ref{exist1}}]  Let $\{g_n\}$ be a sequence of $C^1$ nonnegative functions defined on $\BBR\times\BBR_+$ such that
	\bal g_n(0,0)=g(0,0)=0,\; g_n \leq g_{n+1} \leq g,\; \sup_{\BBR \times \BBR_+}g_n=n \text{ and } \lim_{n \to \infty}\norm{g_n- g}_{L^\infty_{\text{loc}}(\BBR \times \BBR_+)}=0. \eal
	We observe that $\Gl_{g_n} \leq \Gl_g<\infty$ where $\Gl_{g_n}$ is defined as in \eqref{integralcond} with $g$ replaced by $g_n$. Therefore the constant $\varrho_0$ in Lemma \ref{transf} can be chosen to depend on $\Gl_g$ (and $N,\mu,\Gw, \tl k, \tl p,\tl q$), but independent of $n$. Similarly, the constant $t_0$ in Lemma \ref{transf} can be chosen to be dependent on $\Gl_g$ (and also $N,\mu,\Gw,\tl k, \tl p, \tl q$), but independent of $n$.   By Lemma \ref{transf}, for any $\varrho \in (0,\varrho_0)$ and $n \in \BBN$, there exists a solution $v_n \in \CO$ (where $\CO$ is defined in \eqref{O}) of
	\bal \left\{ \BAL -L_\mu v_n &= g_n (v_n + \varrho\BBK_\mu[\nu],|\nabla(v_n + \varrho\BBK_\mu[\nu])|) \quad \text{in } \Gw\setminus\xS, \\ \tr(v_n) &= 0.
	\EAL \right. \eal
	Set $u_n=v_n + \varrho \BBK_\mu[\nu]$ then $\tr(u_n)=\varrho \nu$ and
	\bel{ul13} -\int_{\Gw} u_n L_\gm\zeta {\dd} x= \int_{\Gw} g_n(u_n,|\nabla u_n|) \zeta {\dd} x -
	\varrho \int_{\Gw} \BBK_\gm[\gn]L_\gm \zeta {\dd} x \forevery \zeta \in {\bf X}_\mu(\Gw\setminus\xS). \ee
	Since $\{v_n\} \sbs \CO$, the sequence $\{g_n(v_n + \varrho\BBK_\mu[\nu], |\nabla(v_n +\varrho \BBK_\mu[\nu])|)\}_n$ is uniformly bounded in $L^1(\Gw,\ei)$ and the sequence $\{ \frac{\mu}{d_\xS^2}v_n \}$ is uniformly bounded in $L^{p_1}(G)$ for every compact subset $G \sbs \Gw\setminus\xS$ for some $p_1>0$.
	As a consequence, $\{ \Gd v_n \}_n$ is uniformly bounded in $L^{1}(G)$. By a standard regularity result  for elliptic equations, $\{ v_n\}$ is uniformly bounded in $W^{2,p_2}(G)$ for some $p_2>1$. Consequently, there exists a subsequence, still denoted by $\{v_n \}$, and a function $v$ such that $v_n \to v$ a.e. in $\Gw\setminus\xS$ and $\nabla v_n \to \nabla v$ a.e. in $\Gw\setminus\xS$. Therefore $u_n \to u$ a.e. in $\Gw\setminus\xS$ where $u=v + \varrho \BBK_\mu[\nu]$ and $g_n(u_n,|\nabla u_n|) \to  g(u,|\nabla u|)$ a.e. in $\Gw\setminus\xS$.
	
	We show that $u_n \to u$ in $L^1(\Gw\setminus\xS;\ei)$. Since $\{v_n\}$ is uniformly bounded in $L^{p}(\Gw\setminus\xS;\ei)$ for any $1<p<p_*$, by \eqref{lpweakmartin1}, we derive that $\{u_n\}$ is uniformly bounded in $L^{p}(\Gw\setminus\xS;\ei)$ for any $1<p<p_*$. Due to H\"older inequality, $\{u_n\}$ is equi-integrable in $L^1(\Gw\setminus\xS;\ei)$. We invoke Vitali's convergence theorem to derive that $u_n \to u$ in $L^1(\Gw\setminus\xS;\ei)$.
	
	By Lemma \ref{subcd0gran} and estimate \eqref{estM}, there exists a positive constant $C$ depending only on $N,\mu,\Gw,\Gl_g,\tl k,\tl p,\tl q, t_0$ such that for any Borel set $E \sbs \Gw\setminus\xS$,
	\bel{ul14} \BAL \int_E g_n(u_n,|\nabla u_n|) \ei {\dd} x
	\leq C\int_{\gl}^\infty g(s,s^{\frac{p_*}{q_*}}) s^{-1-p_*} {\dd} s + g(\gl,\gl^{\frac{p_*}{q_*}})\int_E \ei {\dd} x. \\
	\EAL \ee
	Note that the first term on the right hand-side of \eqref{ul14} tends to $0$ as $\gl \to \infty$. Therefore for any $\vge>0$, there exists $\gl>0$ such that the first term on the right hand-side of \eqref{ul14} is smaller than $\frac{\vge}{2}$. Fix such $\gl$ and put $\eta=\vge \left(2\max\{g(\gl,\gl^{\frac{p_*}{q_*}}),1\}\right)^{-1}$. Then, by \eqref{ul14},
	\bal \int_E \ei \dd x \leq \eta \Longrightarrow \int_E g_n(u_n,|\nabla u_n|) \ei \dd x < \vge. \eal
	Therefore the sequence $\{g_n(u_n,|\nabla u_n|)\}$ is equi-integrable in $L^1(\Gw\setminus\xS;\ei)$.
	
	Due to Vitali convergence theorem, we deduce that
	\bal g_n(u_n,|\nabla u_n|) \to  g(u,|\nabla u|)  \quad \text{in } L^1(\Gw\setminus\xS;\ei).
	\eal
	Finally, by sending $n \to \infty$ in each term of \eqref{ul13} we obtain
	\bal -\int_{\Gw} u L_\gm\zeta dx= \int_{\Gw} g(u,|\nabla u|) \zeta dx  - \varrho \int_{\Gw}\BBK_\gm[\gn]L_\gm \zeta dx \forevery \zeta \in {\bf X}_\mu(\Gw\setminus\xS). \eal
	This means $u$ is a nonnegative weak solution of \eqref{sourceprobrho}. Therefore,
\bal
u=\BBG_\mu[g(u,|\nabla u|)] + \varrho\BBK_\mu[\nu] \quad \text{in } \Omega\setminus\xS,
\eal
	which implies that $u \geq \varrho\BBK_\mu[\nu]$ in $\Gw\setminus\xS$. The proof is complete.
\end{proof}

\section{Nonlinear equations with $g(u,|\nabla u|) = |u|^p|\nabla u|^q$} \label{supercritical source}

\subsection{Abstract setting} In this subsection, we present an abstract setting which will be applied to our particular framework in the next subsection.

Let $\mathbf{Z}$ be a metric space and $\gw \in\GTM^+(\mathbf{Z}).$ Let $J : \mathbf{Z} \times \mathbf{Z} \to (0,\infty]$ be a Borel positive kernel such that $J$ is symmetric and $J^{-1}$ satisfies a quasi-metric
inequality, i.e. there is a constant $C>1$ such that for all $x, y, z \in \mathbf{Z}$,
\bal 
J(x,y)^{-1}\leq C\left(J(x,z)^{-1}+J(z,y)^{-1}\right).
\eal
Under these conditions, one can define the quasi-metric $\mathbf{d}$ by
\bal
\mathbf{d}(x,y):=J(x,y)^{-1}
\eal
and denote by $\sB(x,r):=\{y\in\mathbf{Z}:\; \mathbf{d}(x,y)<r\}$ the open $\mathbf{d}$-ball of radius $r > 0$ and center $x$.
Note that this set can be empty.

For $\xo\in\GTM^+(\mathbf{Z})$ and a positive function $\phi$,  the potentials $\BBJ[\gw]$ and $\BBJ[\gf,\gw]$ are defined by
\bal
\BBJ[\gw](x):=\int_{\mathbf{Z}}J(x,y) \dd\xo(y)\quad\text{and}\quad \BBJ[\gf,\gw](x):=\int_{\mathbf{Z}}J(x,y)\gf(y) \dd\gw(y).
\eal
For $t>1$ the capacity $\text{Cap}_{\BBJ,t}^\gw$ in $\mathbf{Z}$ is defined for any Borel $E\subset\mathbf{Z}$ by
\bal
\text{Cap}_{\BBJ,t}^\gw(E):=\inf\left\{\int_{\mathbf{Z}}\gf(x)^{t} \dd\gw(x):\;\;\gf\geq0,\;\; \BBJ[\gf,\gw] \geq\1_E\right\}.
\eal

\begin{proposition}\label{t2.1}  \emph{(\cite{KV})} Let $p>1$ and $\gl,\gw \in\GTM^+(\mathbf{Z})$ such that
	\ba
	\int_0^{2r}\frac{\gw\left(\sB(x,s)\right)}{s^2} \dd s &\leq C\int_0^{r}\frac{\gw\left(\sB(x,s)\right)}{s^2} \dd s ,\label{2.3}\\
	\sup_{y\in \sB(x,r)}\int_0^{r}\frac{\gw\left(\sB(y,r)\right)}{s^2} \dd s&\leq C\int_0^{r}\frac{\gw\left(\sB(x,s)\right)}{s^2} \dd s,\label{2.4}
	\ea
for any $r > 0,$ $x \in \mathbf{Z}$, where $C > 0$ is a constant. Then the following statements are equivalent.
	
	1. The equation $v=\BBJ[|v|^p,\gw]+\ell \BBJ[\gl]$ has a positive solution for $\ell>0$ small.
	
	2. For any Borel set $E \subset \mathbf{Z}$, there holds
	$
\int_E \BBJ[\1_E\gl]^p \dd \gw \leq C\, \gl(E).
	$
	
	3. The following inequality holds
	$
\BBJ[\BBJ[\gl]^p,\gw]\leq C\BBJ[\gl]<\infty\quad \gw-a.e.
$

	4. For any Borel set $E \subset \mathbf{Z}$ there holds
$
\gl(E)\leq C\, \emph{Cap}_{\BBJ,p'}^\gw(E).
$
\end{proposition}

\subsection{Necessary and sufficient conditions for the existence}
For $\xa\leq N-1$ and $0<\sigma<N$, set
\ba \label{Nas}
\CN_{\ga,\sigma}(x,y)&:=\frac{\max\{|x-y|,d_\xS(x),d_\xS(y)\}^{\xa}}{|x-y|^{N-\sigma}\max\{|x-y|,d(x),d(y)\}^{\sigma}},\quad (x,y)\in\overline{\xO}\times\overline{\xO}, x \neq y, \\ \nonumber \BBN_{\alpha,\sigma}[\omega](x)&:=\int_{\overline{\Gw}} \CN_{\alpha,\sigma}(x,y) \dd\omega(y), \quad \omega \in \GTM^+(\overline \Gw).
\ea

It can be seen below that $\BBN_{\alpha,1}$ with $\dd \gw=d(x)^b d_\xS(x)^\theta \1_{\Omega \setminus \Sigma}(x)\,\dx$ satisfies all assumptions of $\BBJ$  in Proposition \ref{t2.1}, for some appropriate $b,\theta\in \BBR$. Let us first prove the quasi-metric inequality.

\begin{lemma}\label{ineq}
Let $\xa\leq N-1$. There exists a positive constant $C=C(N,\xO,\Sigma,\xa)$ such that
\bal \CN_{\alpha,1}(x,y)^{-1}\leq C\left(\CN_{\alpha,1}(x,z)^{-1}+\CN_{\alpha,1}(z,y)^{-1}\right),\quad \forall x,y,z\in \overline{\xO}.
\eal
\end{lemma}
\begin{proof}
The proof is very similar to the one in \cite[Lemma 6.4]{GkiNg_source} and we omit it.
\end{proof}

Next we give sufficient conditions for \eqref{2.3} and \eqref{2.4} to hold.
\begin{lemma}[{\cite[Lemma 6.5]{GkiNg_source}}]\label{l2.3} Let $b>0$, $\theta>k-N$ and  $\dd \gw=d(x)^b d_\xS(x)^\theta \1_{\xO \setminus \Sigma}(x)\,\dx$. Then
\bal
\gw(B(x,s))\asymp \max\{d(x),s\}^b\max\{d_\xS(x),s\}^{\theta} s^N, \; \text{for all } x\in\xO \text{ and } 0 < s\leq 4 \mathrm{diam}(\xO).
\eal
\end{lemma}

\begin{lemma}\label{vol}
Let $\xa< N-1$, $b> 0$, $\theta>\max\{k-N,-1-\xa\}$ and $\dd \gw=d(x)^b d_\xS(x)^\theta$ $\1_{\Omega \setminus \Sigma}(x)\,\dx$.  Then \eqref{2.3}  holds.
\end{lemma}
\begin{proof}
We note that if $s\geq (4\diam(\xO))^{N-\xa}$ then $\xo(\sB(x,s))=\xo(\overline{\xO})<\infty$, where $\sB(x,s)=\{y \in \Omega \setminus \Sigma: \mathbf{d}(x,y)<s\}$ and $\mathbf{d}(x,y)= \CN_{\xa,1}(x,y)^{-1}$.

We first assume that $0<\xa<N-1$.  Let $x\in \Sigma_\frac{\xb_0}{4}$ then
\ba \label{C0d}
0<C_0 \leq d(x)\leq 2 \diam(\xO),
\ea
where $C_0$ depends on $\xO,\Sigma,\beta_0$. Set
\bal
\sC(x,s):=\left\{y \in \Omega \setminus \Sigma: \frac{|x-y|^{N-1}}{\max\{|x-y|,d_\xS(x),d_\xS(y)\}^\xa}< s\right\}.
\eal
Then
\ba \label{CB} \sC\left(x,\frac{s}{2\diam(\xO)}\right)\subset\sB(x,s)\subset \sC\left(x,\frac{s}{C_0}\right).
\ea
We note that $B(x,S_1)\subset\sC(x,s)\subset B(x,l_1S_1)$
where $S_1=\max\{s^{\frac{1}{N-1-\xa}},s^{\frac{1}{N-1}}d_\xS(x)^{\frac{\xa}{N-1}} \}$ and $l_1=2^{\frac{\xa}{N-1-a}}$. Therefore, by Lemma \ref{l2.3}, we obtain
\ba\nonumber
\xo(\sB(x,s))&\asymp \max\left\{d_\xS(x),\max\{s^{\frac{1}{N-1-\xa}},s^{\frac{1}{N-1}}d_\xS(x)^{\frac{\xa}{N-1}} \}\right\}^\theta \max\{s^{\frac{1}{N-1-\xa}},s^{\frac{1}{N-1}}d_\xS(x)^{\frac{\xa}{N-1}} \}^N\\
&\asymp d_\xS(x)^{\theta+\frac{\xa N}{N-1}}s^{\frac{N}{N-1}}\1_{(0,d_\xS(x)^{N-1-\xa})}(s) +  s^{\frac{\theta+N}{N-1-\xa}} \1_{[d_\xS(x)^{N-1-\xa},M)}(s) + \1_{[M,\infty)}(s),
\label{xoest1}
\ea
where
\bal
M:=(4\diam(\xO))^\frac{N(N-\xa)}{b+N}+(4\diam(\xO))^{\frac{(N-2)(N-\xa)}{N}}+(4\diam(\xO))^{\frac{(N-\xa-1)(N-\xa)}{\theta +N}}.
\eal

Next we assume that $\xa\leq0$. Let $x\in \Sigma_\frac{\xb_0}{4}$ then \eqref{C0d} and \eqref{CB} hold. We also have $B(x,{l_2S_2})\subset\sC(x,s)\subset B(x,S_2)$, where $S_2=\min\{s^{\frac{1}{N-1-\xa}},s^{\frac{1}{N-1}}d_\xS(x)^{\frac{\xa}{N-1}}\}$ and $l_2=2^{\frac{\xa}{N-1}}$. Therefore by Lemma \ref{l2.3}, we obtain
\bal
\xo(\sB(x,s))&\asymp \max\left\{d_\xS(x),\min\{s^{\frac{1}{N-1-\xa}},s^{\frac{1}{N-1}}d_\xS(x)^{\frac{\xa}{N-1}}\}\right\}^\theta \min\{s^{\frac{1}{N-1-\xa}},s^{\frac{1}{N-1}}d_\xS(x)^{\frac{\xa}{N-1}}\}^N\\
&\asymp d_\xS(x)^{\theta+\frac{\xa N}{N-1}}s^{\frac{N}{N-1}} \1_{(0,d_\xS(x)^{N-1-\xa})}(s) + s^{\frac{\theta+N}{N-1-\xa}}\1_{[d_\xS(x)^{N-1-\xa},M)}(s) + \1_{[M,\infty)}(s).
\eal

Next consider $x\in \xO_{\frac{\xb_0}{4}},$ then there exists a positive constant $C_3=C_3(\xO,\Sigma,\xa,\xb_0)$ such that $C_3\leq d_\xS(x)< 2 \diam(\xO).$ Set
\bal
\sE(x,s):=\{y \in \Omega \setminus \Sigma: |x-y|^{N-1}\max\{|x-y|,d(x),d(y)\}< s\}.
\eal
We see that
\bal
\sE(x,\min\{C_3^\xa,2^\xa\diam(\xO)^\xa\}s)\subset\sB(x,s)\subset \sE(x,\max\{C_3^\xa,2^\xa\diam(\xO)^\xa\}s).
\eal
We also have
$B(x,l_3S_3)\subset\sE(x,s)\subset B(x,S_3)$, where $S_3=\min\{s^{\frac{1}{N}},s^{\frac{1}{N-1}}d(x)^{-\frac{1}{N-1}}\}$ and $l_3=2^{-\frac{1}{N-2}}$. Again, by Lemma \ref{l2.3}, we obtain
\bal
\xo(\sB(x,s))&\asymp \max\left\{d(x),\min\{s^{\frac{1}{N}},s^{\frac{1}{N-1}}d(x)^{-\frac{1}{N-1}}\}\right\}^b \min\{s^{\frac{1}{N}},s^{\frac{1}{N-1}}d(x)^{-\frac{1}{N-1}}\}^N\\
&\asymp d(x)^{b-\frac{N}{N-1}}s^{\frac{N}{N-1}}\1_{(0,d(x)^{N})}(s) + s^{\frac{b+N}{N}}\1_{[d(x)^{N},M)}(s) + \1_{[M,\infty)}(s).
\eal

Let $0<\bar \xb \leq \frac{\xb_0}{4}$ and $x\in \xO\setminus (\xO_{\bar \xb}\cup \Sigma_{\bar \xb}).$ Then there exists a positive constant $C_4=C_4(\xO,\Sigma,\bar \xb)$ such that $C_4\leq d_\xS(x),d(x)< 2 \diam(\xO).$ By Lemma \ref{l2.3}, we can show that
\ba \label{xoest4}
\xo(\sB(x,s)) \asymp s^{\frac{N}{N-2}}\1_{(0,M)}(s) + \1_{[M,\infty)}(s).
\ea

Combining \eqref{xoest1}--\eqref{xoest4} leads to \eqref{2.3}. The proof is complete.
\end{proof}

\begin{lemma}\label{vol-2}
	We assume that $\xa< N-1$, $b> 0$, $\theta>\max\{k-N,-1-\xa\}$ and $\dd \gw=d(x)^b d_\xS(x)^\theta\1_{\Omega \setminus \Sigma}(x)\,\dx$.  Then \eqref{2.4}  holds.
\end{lemma}
\begin{proof}
We  consider only the case $\xa>0$  and $x\in \Sigma_{\frac{\xb_0}{16}}$ since the other cases $x\in \xO_{\frac{\xb_0}{16}}$ and $x\in \xO\setminus(\xO_{\frac{\xb_0}{16}}\cup \Sigma_{\frac{\xb_0}{16}})$ can be treated similarly and we omit them. We take $r>0$.

\noindent \textbf{Case 1:} $0<r<(2\diam(\xO))^{-\xa}\left(\frac{\xb_0}{16}\right)^{N}$. In this case, we note that $\sB(x,r) \subset \Sigma_{\frac{\xb_0}{8}}$. This and  \eqref{xoest1} imply that, for any $y \in \sB(x,r)$,
\ba \nonumber
\int_0^r\frac{\xo(\sB(y,s))}{s^2} \dd s\asymp d_\xS(y)^{\theta+\frac{\xa N}{N-1}}r^{\frac{1}{N-1}}\1_{(0,d_\xS(y)^{N-1-\xa})}(r) + r^{\frac{\theta+1+\xa}{N-1-\xa}}\1_{[d_\xS(y)^{N-1-\xa},M)}(r) + \1_{[M,\infty)}(r).
\ea

If $|x-y|\leq \frac{1}{2}d_\xS(x)$ then $\frac{1}{2}d_\xS(x)\leq d_\xS(y)\leq \frac{3}{2}d_\xS(x)$. Therefore, when $d_\xS(y),d_\xS(x)\geq r^{\frac{1}{N-1-\xa}}$, we obtain
\bal
\int_0^r\frac{\xo(\sB(y,s))}{s^2} \dd s \asymp d_\xS(y)^{\theta+\frac{\xa N}{N-1}}r^{\frac{1}{N-1}}\asymp d_\xS(x)^{\theta+\frac{\xa N}{N-1}}r^{\frac{1}{N-1}}\asymp \int_0^r\frac{\xo(\sB(x,s))}{s^2} \dd s.
\eal
If $d_\xS(y)\geq r^{\frac{1}{N-1-\xa}}$ and $d_\xS(x)\leq r^{\frac{1}{N-1-\xa}}$ then $d_\xS(y)\leq \frac{3}{2}r^{\frac{1}{N-1-\xa}}$, which implies
\bal
\int_0^r\frac{\xo(\sB(y,s))}{s^2} \dd s \asymp d_\xS(y)^{\theta+\frac{\xa N}{N-1}}r^{\frac{1}{N-1}}\asymp r^{\frac{\theta+1+\xa}{N-1-\xa}}\asymp \int_0^r\frac{\xo(\sB(x,s))}{s^2} \dd s.
\eal
If $d_\xS(y)\leq r^{\frac{1}{N-1-\xa}}$ and $d_\xS(x)\geq r^{\frac{1}{N-1-\xa}}$ then $d_\xS(x)\leq 2r^{\frac{1}{N-1-\xa}}$, which yields
\bal
\int_0^r\frac{\xo(\sB(x,s))}{s^2} \dd s\asymp d_\xS(x)^{\theta+\frac{\xa N}{N-1}}r^{\frac{1}{N-1}}\asymp r^{\frac{\theta+1+\xa}{N-1-\xa}}\asymp \int_0^r\frac{\xo(\sB(y,s))}{s^2} \dd s.
\eal
If $d_\xS(y)\leq r^{\frac{1}{N-2-\xa}}$ and $d_\xS(x)\leq r^{\frac{1}{N-2-\xa}}$ then
\bal
\int_0^r\frac{\xo(\sB(x,s))}{s^2} \dd s\asymp r^{\frac{\theta+1+\xa}{N-1-\xa}}\asymp \int_0^r\frac{\xo(\sB(y,s))}{s^2} \dd s.
\eal

Now we assume that $y \in \sB(x,r)$ and $|x-y|\geq \frac{1}{2}d_\xS(x).$ Then
\bal d_\xS(y)\leq \frac{3}{2}|x-y| \quad \text{and} \quad
|x-y|\leq C(N,\xO,\xS,\xb_0) r^{\frac{1}{N-\xa-1}}.
\eal
Hence $d_\xS(x),d_\xS(y)\lesssim r^{\frac{1}{N-\xa-1}}$. Proceeding as above we obtain the desired result. \medskip

\noindent \textbf{Case 2:} $r\geq \left(\frac{\xb_0}{16}\right)^{N}(2\diam(\xO))^\xa$. By \eqref{xoest1}--\eqref{xoest4}, we can prove that
\bal
\int_0^r\frac{\xo(\sB(y,s))}{s^2} \dd s \asymp 1,\quad \forall y\in \overline{\xO},
\eal
and the desired result follows easily in this case.
\end{proof}


For $\alpha \leq N-1$, $0<\sigma<N$, $b>0$, $\theta>-N+k$ and $s>1$, define the capacity 
\bal \text{Cap}_{\BBN_{\alpha,\sigma},s}^{b,\theta}(E) :=\inf\left\{\int_{\overline{\xO}}d^b d^\theta_\xS\gf^s\,\dx:\;\; \gf \geq 0, \;\;\BBN_{\alpha,\sigma}[ d^b d_{\Sigma}^\theta\gf ]\geq\1_E\right\} \quad \text{for  Borel set } E\subset\overline{\xO}.
\eal
Here $\1_E$ denotes the indicator function of $E$. Furthermore, by \cite[Theorem 2.5.1]{Ad},
\be\label{dualcap}
(\text{Cap}_{\BBN_{\alpha,\sigma},s}^{b,\theta}(E))^\frac{1}{s}=\sup\{\tau(E):\tau\in\GTM^+(E), \|\Nthb[\tau]\|_{L^{s'}(\xO;d^b d^\theta_\xS)} \leq 1 \}.
\ee

If $\xn\in\mathfrak{M}^+(\partial\xO\cup\xS),$ we extent $\xn$ in the whole $\overline{\xO}$ by setting $\xn(\xO\setminus \xS)=0.$ By applying Proposition \ref{t2.1} with $J(x,y)=\CN_{2\am+1,1}(x,y)$, $\dd \xo= d(x)^{p+1}d_{\xS}(x)^{-\am(p+1)-(\am+1)q}\dx$ and $\dd \lambda = \dd \xn$, and taking into account that \eqref{2.3} and \eqref{2.4} are satisfied due to Lemma \ref{vol} and Lemma \ref{vol-2} respectively, we obtain the following result.

\begin{theorem}\label{sourceth}   Let $\xm<\frac{(N-2)^2}{4}$, $\xn\in\mathfrak{M}^+(\partial\xO\cup\xS)$ and $p \geq 0$, $q\geq0$ such that $p+q>1$ and
$\am p+(\am+1)q<2+\am$.
Then the following statements are equivalent.
	
	1. For $\ell>0$ small, the following equation has a positive solution 
\bal
v=\BBN_{2\am+1,1}[|v|^{p+q}d^{p+1}d_{\xS}^{-\am(p+1)-(\am+1)q}]+\ell \BBN_{2\am+1,1}[\xn].
\eal

	2. For any Borel set $E \subset \overline{\xO}$, there holds
\ba\label{con1}
\int_E \BBN_{2\am+1,1}[\1_E\xn](x)^{p+q} d(x)^{p+1}d_{\xS}(x)^{-\am(p+1)-(\am+1)q}\dx \leq C\, \xn(E).
\ea
	
	3. The following inequality holds
\ba\label{con2}
\BBN_{2\am+1,1}\bigg[\BBN_{2\am+1,1}[\xn]^{p+q}d^{p+1}d_{\xS}^{-\am(p+1)-(\am+1)q}\bigg]\leq C\BBN_{2\am+1,1}[\xn]<\infty\quad \text{a.e. in } \Omega.
\ea

	4. For any Borel set $E \subset \overline{\xO}$ there holds
\bal 
\xn(E)\leq C\, \mathrm{Cap}_{\BBN_{2\am+1,1},(p+q)'}^{p+1,-\am(p+1)-(\am+1)q}(E).
\eal
\end{theorem}

\begin{remark}
Condition $\am p+(\am+1)q<2+\am$ is required to ensure that $-\am(p+1)-(\am+1)q > \max\{k-N,-2-2\am\}=-2-2\am$, which in turn allows to apply Lemmas \ref{vol} and \ref{vol-2}.	
\end{remark}

Now we are ready to give

\begin{proof}[{\textbf{Proof of Theorem \ref{existth}}}]
Define 
\bal
{\mathscr V}:=\{v \in W_{\mathrm{loc}}^{1,1}(\Omega) \text{ such that }  \|v\|_{{\mathscr V}}:=\| v \|_{L^{p+q}(\Gw,d^{-q+1}d_\xS^{-\am-q})} + \|  \nabla v \|_{L^{p+q}(\Gw,d^{p+1}d_\xS^{p-\am})}<+\infty\}.
	\eal
For $u \in{\mathscr V}$, put
	\bal
{\mathds H}[u](x):=\BBG_\mu[|u|^p|\nabla u|^q](x) + \BBK_\mu[\varrho \nu](x) \quad \text{for a.e. } x \in \Gw.
	\eal
From \eqref{Greenesta}, \eqref{Martinest1}, \eqref{gradientgreen} and \eqref{est:gradK-K}, we have 
\be\label{GN2a}\BAL
G_\mu(x,y)&\asymp d(x)d(y)(d_{\xS}(x)d_{\xS}(y))^{-\am} \CN_{2\am,2}(x,y) \\
&\lesssim d(x)d(y)(d_{\xS}(x)d_{\xS}(y))^{-\am} \CN_{2\am+1,1}(x,y) \quad \forall x,y \in \Omega \setminus \Sigma, x \neq y,
\EAL
\ee
\bal
|\nabla_x G_\mu(x,y)|\lesssim d_{\xS}(x)^{-\am-1}d(y)d_{\xS}(y)^{-\am} \CN_{2\am+1,1}(x,y) \quad \forall x,y \in \Omega \setminus \Sigma, x \neq y,
\eal
\ba \label{KN2a1}
K_\mu(x,z)\asymp d(x)d_{\xS}(x)^{-\am} \CN_{2\am+1,1}(x,z) \quad \forall x \in \Omega \setminus \Sigma, z \in \partial\xO\cup\Sigma,
\ea
and
\ba \label{KN2a1gran}
|\nabla_x K_\mu(x,z)|\lesssim d_{\xS}(x)^{-\am-1} \CN_{2\am+1,1}(x,z) \quad \forall x \in \Omega \setminus \Sigma, z \in \partial\xO\cup\Sigma.
\ea

From \eqref{GN2a}-\eqref{KN2a1gran}, we obtain
	\bal
	|{\mathds H}[u](x)| &\leq C_1d(x) d_{\xS}(x)^{-\am} \BBN_{2\am+1,1}[d d_{\xS}^{-\am}|u|^p|\nabla u|^q](x) + C_1 d(x)d_{\xS}(x)^{-\am} \BBN_{2\am+1,1}[\varrho \nu](x), \\
	|\nabla {\mathds H}[u](x)| &\leq C_1  d_{\xS}(x)^{-\am-1} \BBN_{2\am+1,1}[d d_{\xS}^\am  |u|^p|\nabla u|^q](x) + C_1 d_{\xS}(x)^{-\am-1}  \BBN_{2\am+1,1}[\varrho \nu](x).
	\eal

	Put
	\bal
{\mathscr F}:=\{ u \in W_{\mathrm{loc}}^{1,1}(\Omega): |u| \leq  2C_1 d(x)d_{\xS}(x)^{-\am} \BBN_{2\am+1,1}[\varrho \nu], \;\; |\nabla u| \leq 2C_1d_{\xS}(x)^{-\am-1}  \BBN_{2\am+1,1}[\varrho \nu]   \}.
	\eal
	
Since condition \eqref{con3-a} holds, it follows from Theorem \ref{sourceth} that \eqref{con1} and \eqref{con2} also hold.  In addition, from \eqref{con1} we have  ${\mathds H}({\mathscr F})\subset {\mathscr V}$ and from \eqref{con2} we infer that there exists $\varrho_0=\varrho_0(p,q,C_1,C)>0$ such that if $\varrho \in (0,\varrho_0)$ then ${\mathds H}({\mathscr F}) \subset {\mathscr F}$.
	
We see that ${\mathscr F}$ is convex and closed under the strong topology of ${\mathscr V}$. Moreover, it can be justified that ${\mathds H}$ is a continuous and compact operator. Therefore, by invoking Schauder the fixed point theorem, we conclude that there exists $u \in {\mathscr F}$ such that ${\mathds H}[u]=u$. Therefore, $u\in {\mathscr F}$ is a weak solution of problem \eqref{sourceprobrho}.
\end{proof}

\begin{proof}[{\textbf{Proof of proposition \ref{th:existence-source-3}}}] In view of Theorem \ref{sourceth} and Theorem \ref{existth}, it is sufficient to prove that \eqref{con1} holds. 
Let $E\subset\overline{\xO}$ be a Borel set. By \eqref{KN2a1}, we have
\ba\BAL\label{1}
\int_E \BBN_{2\am+1,1}[\1_E\xn]^{p+q} d^{p+1}d_{\xS}^{-\am(p+1)-(\am+1)q}\dx
&\asymp \int_E dd_{\xS}^{-\am}\left((dd_{\xS})^{-\frac{q}{p+q}}\mathbb{K}_\mu[\1_E\xn]\right)^{p+q}\,\dx\\
&\leq \int_\xO dd_{\xS}^{-\am}\left((dd_{\xS})^{-\frac{q}{p+q}}\mathbb{K}_\mu[\1_E\xn]\right)^{p+q}\,\dx.
\EAL
\ea

(i) Assume $\nu$ is concentrated on $\Sigma$. Using estimate \eqref{estgranmartin2} with $\xg=\frac{q}{p+q}$, we derive
\bal 
\int_\xO dd_{\xS}^{-\am}\left((dd_{\xS})^{-\frac{q}{p+q}}\mathbb{K}_\mu[\1_E\xn]\right)^{p+q}\,\dx\leq C(\xO,\xS,\xm,p,q) \xn(E \cap \Sigma)^{p+q}\leq C\xn(\Sigma)^{p+q-1}\xn(E).
\eal
This and \eqref{1} imply the desired result.

(ii) Assume $\nu$ is concentrated on $\partial \Omega$. Thanks to a similar argument as in (i) and using estimate \eqref{estgranmartin1}, we obtain the desired result. We leave the  detail to the reader.
\end{proof}

Similarly we may prove the following proposition.

\begin{proposition}
Let $\xm<\frac{(N-2)^2}{4}$ $p,q\geq0$, $p+q>1$, 
$\am p+(\am+1)q<2+\am$ and $\xn\in\mathfrak{M}^+(\partial\xO\cup\xS)$. Assume that one of the following conditions holds

(i) $\am\leq-1$ and $(N-\am-2)p+(N-\am-1)q<N-\am$,

(ii) If $\am> -1$ and $(N-1)p+Nq<N+1$.

\noindent Then for $\varrho>0$ small enough, problem \eqref{problem:source-power} admits a nonnegative weak solution.	
\end{proposition}

In the sequel we will study when condition \eqref{con3-a} is valid in terms of Bessel capacities. For $\theta \in\BBR$ we define the Bessel kernel of order $\theta$ in $\R^d$ by
\bal
\CB_{d,\theta}(\xi):=\CF^{-1}\left((1+|.|^2)^{-\frac{\theta}{2}}\right)(\xi),
\eal
where $\CF$ is the Fourier transform in the space $\mathcal{S}'(\R^d)$ of moderate distributions in $\BBR^d$. Set
\bal
\BBB_{d,\theta}[\lambda](x):= \int_{\R^d}\CB_{d,\theta}(x-y) \dd\lambda(y), \quad x \in \R^d, \quad \lambda \in \GTM(\R^d).
\eal
Let
$L_{\theta,\kappa}(\BBR^d):=\{f=\CB_{d,\theta} \ast g:g\in L^{\kappa}(\BBR^d)\}
$ be the Bessel space with the norm
\bal
\|f\|_{L_{\theta,\kappa}}:=\|g\|_{L^\kappa}=\|\CB_{d,-\theta}\ast f\|_{L^\kappa}.
\eal

The Bessel capacity is defined as follows.

\begin{definition}
 Let $1< \kappa <\infty$ and $E\subset\BBR^d.$ Set
 \bal
 \mathcal{S}_E:=\{g\in L^\kappa(\BBR^d):\;g\geq0,\;(\CB_{d,\theta} \ast g)(x)\geq 1\;\;\text{for any}\;x\in E\}.
 \eal
Then
\bal
\mathrm{Cap}_{{\CB_{d,\theta},\kappa}}^{\R^d}(E):=\inf\{\|g\|^\kappa_{L^\kappa(\BBR^d)}:\; g\in \mathcal{S}_E \}.
\eal
If $\mathcal{S}_E=\emptyset,$ we set $\mathrm{Cap}_{{\CB_{d,\theta},\kappa}}^{\R^d}(E)=\infty.$
\end{definition}

 If $\Gamma \subset \overline{\Omega}$ is a $C^2$ submanifold without boundary, of dimension $d$ with $1 \leq d \leq N-1$ then there exist open sets $O_1,...,O_m$ in $\BBR^N$, diffeomorphism $T_i: O_i \to B^{d}(0,1)\times B^{N-d}(0,1) $ and compact sets $K_1,...,K_m$ in $\Gamma$ such that

(i) $K_i \sbs O_i$, $1 \leq i \leq m$ and $ \Gamma= \cup_{i=1}^m K_i$,

(ii) $T_i(O_i \cap \Gamma)=B_1^{d}(0) \times \{ x'' = 0_{\mathbb{R}^{N-d}} \}$, $T_i(O_i \cap \Gw)=B_1^{d}(0)\times B_1^{N-d}(0)$,

(iii) For any $x \in O_i \cap (\xO\setminus \Gamma)$, there exists $y \in O_i \cap  \xS$ such that $d_\Gamma(x)=|x-y|$ (here $d_\Gamma(x)$ denotes the distance from $x$ to $\Gamma$). \smallskip

We then define the $\mathrm{Cap}_{\gth,s}^{\Gamma}-$capacity of a compact set $E \sbs \Gamma$ by
\bel{Capsub} \mathrm{Cap}_{\gth,\kappa}^{\Gamma}(E):=\sum_{i=1}^m \mathrm{Cap}_{\CB_{d,\gth},\kappa}^{\mathbb{R}^d}(\tl T_i(E \cap K_i)), \ee
where $T_i(E \cap K_i)=\tl T_i(E \cap K_i) \times  \{ x'' = 0_{\mathbb{R}^{N-d}} \}$. We remark that the definition of the capacities does not depends on $O_i$.

Note that if $\theta \kappa > d$ then
$\mathrm{Cap}_{\gth,\kappa}^{\Gamma}(\{z\})>C>0$ for all $z \in \Gamma$.

\begin{lemma}\label{besov}Let $k\geq1,$ $p,q\geq0$ be such that $p+q>1$. Assume \eqref{3condition} holds. Let  $\xn\in \mathfrak{M}^+(\mathbb{R}^k)$ with compact support in $B^k(0,\frac{R}{2})$ for some $R>0$ and $\vartheta$ be as in \eqref{vartheta}.
	For $x \in \mathbb{R}^{k+1}$, we write $x=(x_1,x') \in \mathbb{R} \times \mathbb{R}^{k}$.
	Then
	\ba  \nonumber
	&\int_{B^k(0,R)}\int_{0}^R x_{1}^{N-k-1-\am(p+1)-(\am+1)q}\left(\int_{B^k(0,R)}\left(|x_1|+|x'-y'|\right)^{-(N-2\am-2)}\dd \nu(y')\right)^{p+q}\dd x_1\dd x'\\ 
	&\asymp \int_{\mathbb{R}^k}\BB_{k,\vartheta}[\xn](x')^{p+q}\, \dd x'. \label{est-Bnu0}
	\ea
	Here the implicit constants depend only on $R,N,k,\mu,p$, and 
	\bal \BBB_{k,\vartheta}[\xn](x') :=\int_{\mathbb{R}^k}\CB_{k,\vartheta}(x'-y')\, \dd \nu(y').
	\eal
\end{lemma}
\begin{proof}
The proof is inspired by the idea in \cite[Proposition 2.8]{BHV} and is split into two steps.

\noindent \textbf{Step 1:} We will prove the upper bound in \eqref{est-Bnu0}. Let $0<x_1<R$ and $|x'|<R$. In view of the proof of \cite[Lemma 3.1.1]{Ad}, we obtain
	\bal \BAL
	&\int_{B^k(0,R)}\left(x_1+|x'-y'|\right)^{-(N-2\am-2)}\,\dd \nu(y')\leq \int_{B^k(x',2R)}\left(x_1+|x'-y'|\right)^{-(N-2\am-2)}\,\dd \nu(y')\\
	&=(N-2\am-2)\left(\int_0^{2R}\frac{\xn(B^k(x',r))}{(x_1+r)^{N-2\am-2}}\frac{\dd r}{x_1+r}+\frac{\xn(B^k(x',2R))}{(x_1+2R)^{N-2\am-2}}\right)\\
	&\lesssim \int_0^{3R}\frac{\xn(B^k(x',r))}{(x_1+r)^{N-2\am-2}}\frac{\dd r}{x_1+r}
	\leq \int_{x_1}^{4R}\frac{\xn(B^k(x',r))}{r^{N-2\am-2}}\frac{\dd r}{r}.
	\EAL
	\eal
Set $\xb=\am(p+1)+(\am+1)q$. We infer from the above estimate that
	\bal
	&\int_{0}^R x_{1}^{N-k-1-\xb}\left(\int_{B^k(0,R)}\left(|x_1|+|x'-y'|\right)^{-(N-2\am-2)}\dd \nu(y')\right)^{p+q}\, \dd x_1\\
	&\lesssim \int_{0}^R x_{1}^{N-k-1-\xb}\left(\int_{x_1}^{4R}\frac{\xn(B^k(x',r))}{r^{N-2\am-2}}\frac{\dd r}{r}\right)^{p+q} \, \dd x_1.
	\eal
	Due to the first condition in \eqref{3condition}, $\beta<N-k$. Let $\xe$ be such that $0<\xe<N-k-\xb$. By H\"older's inequality and Fubini's theorem, we have
	\bal
	\int_{0}^R &x_{1}^{N-k-1-\xb}\left(\int_{x_1}^{4R}\frac{\xn(B^k(x',r))}{r^{N-2\am-2}}\frac{\dd r}{r}\right)^{p+q} \dd x_1\\
	&\leq \int_{0}^R x_{1}^{N-k-1-\xb}\left(\int_{x_1}^\infty r^{-\frac{\xe (p+q)'}{p+q}}\frac{\dd r}{r}\right)^{\frac{p+q}{(p+q)'}}\int_{x_1}^{4R}\left(\frac{\xn(B^k(x',r))}{r^{N-2\am-2-\frac{\xe}{p+q}}}\right)^{p+q}\frac{\dd r}{r}\, \dd x_1\\
	&=C(p,q,\xe)\int_{0}^R x_{1}^{N-k-1-\xb-\xe}\int_{x_1}^{4R}\left(\frac{\xn(B^k(x',r))}{r^{N-2\am-2-\frac{\xe}{p+q}}}\right)^{p+q}\frac{\dd r}{r}\, \dd x_1\\
	&\leq C(p,q,\xe,N,k,\am,R)\int_0^{4 R} \left(\frac{\xn(B^k(x',r))}{r^{N-2\am-2-\frac{N-k-\xb}{p+q}}}\right)^{p+q}\frac{\dd r}{r}.
	\eal
	By the second and the first condition in \eqref{3condition} and the definition of $\vartheta$ in \eqref{vartheta}, we see that $0 < \vartheta <k$. Keeping in my that
\bal
	N-2\am-2-\frac{N-k-\xb}{p+q} = k - \vartheta,
\eal
we have
	\bal
	\int_0^{4 R} \left(\frac{\xn(B^k(x',r))}{r^{k-\vartheta}}\right)^{p+q}\frac{\dd r}{r}
	&=\sum_{n=0}^\infty\int_{2^{-n+1}R}^{2^{-n+2}R}\left(\frac{\xn(B^k(x',r))}{r^{k-\vartheta}}\right)^{p+q}\frac{\dd r}{r}\\ 
	&\leq \ln2\left( \sum_{n=0}^\infty 2^{(n-1)(k-\vartheta)}\frac{\xn(B^k(x',2^{-n+2}R))}{R^{k- \vartheta}}\right)^{p+q}\\
	&\leq  2^{(p+q)(k-\vartheta)}(\ln2)^{-(p+q-1)}\left(\int_0^{8 R}\frac{\xn(B^k(x',r))}{r^{k-\vartheta}}\frac{\dd r}{r}\right)^{p+q}.
	\eal
	This and \cite[Theorem 2.3]{BNV} (note that conditions on $p,q$ allow to use \cite[Theorem 2.3]{BNV}) imply
	\bal 
	&\int_{B^k(0,R)}\int_{0}^R x_{1}^{N-k-1-\xb\am}\left(\int_{B^k(0,R)}\left(x_1+|x'-y'|\right)^{-(N-2\am-2)}\, \dd \nu(y')\right)^{p+q}\, \dd x_1\, \dd x'\\
	&\lesssim \int_{\mathbb{R}^k}\BB_{k,\vartheta}[\xn](x')^{p+q}\, \dd x',
	\eal
which leads to the upper bound in \eqref{est-Bnu0}. \medskip

\noindent \textbf{Step 2:} We will prove the lower bound in \eqref{est-Bnu0}. Let $0<x_1<R$ and $|x'|<R$. Then by \cite[Lemma 3.1.1]{Ad}, we have
	\bal
	\int_{B^k(0,R)}\left(x_1+|x'-y'|\right)^{-(N-2\am-2)}\,\dd \nu(y')
	&=
	(N-2\am-2)\int_{x_1}^{\infty}\frac{\xn(B^k(x',r-x_1))}{r^{N-2\am-2}}\frac{\dd r}{r}\\
	&\geq C(N,\am)\int_{x_1}^{\infty}\frac{\xn(B^k(x',r))}{r^{N-2\am-2}}\frac{\dd r}{r}.
	\eal
	It follows that
	\bal
	&\int_{0}^R x_{1}^{N-k-1-\xb}\left(\int_{B^k(0,R)}\left(x_1+|x'-y'|\right)^{-(N-2\am-2)}\dd \nu(y')\right)^{p+q} \, \dd x_1\\
	&\gtrsim \int_{0}^R x_{1}^{N-k-1-\xb}\left(\int_{x_1}^\infty\frac{\xn(B^k(x',r))}{r^{N-2\am-2}}\frac{\dd r}{r}\right)^{p+q} \, \dd x_1\\
	&\gtrsim \int_{0}^R\left(\frac{\xn(B^k(x',x_1))}{x_1^{k-\vartheta}}\right)^{p+q}\frac{\dd x_1}{x_1}.
	\eal
	For $0<r<\frac{R}{2}$, we obtain
	\bal
	\int_{0}^R\left(\frac{\xn(B^k(x',x_1))}{x_1^{k-\vartheta}}\right)^{p+q}\frac{\dd x_1}{x_1} \geq \int_{r}^{2r}\left(\frac{\xn(B^k(x',x_1))}{x_1^{k-\vartheta}}\right)^{p+q}\frac{\dd x_1}{x_1} \gtrsim \left(\frac{\xn(B^k(x',r))}{r^{k-\vartheta}}\right)^{p+q},
	\eal
	which implies
	\bal
	\int_{0}^R\left(\frac{\xn(B^k(x',x_1))}{x_1^{k-\vartheta}}\right)^{p+q}\frac{\dd x_1}{x_1}
	\gtrsim \left(\sup_{0<r<\frac{R}{2}}\frac{\xn(B^k(x',r))}{r^{k-\vartheta}}\right)^{p+q}.
	\eal
	This, together with the assumption that $\xn$ has compact support in $B(0,\frac{R}{2})$ and \cite[Theorem 2.3]{BNV}, yields
	\bal
	&\int_{B^k(0,R)}\int_{0}^R x_{1}^{N-k-1-\xb}\left(\int_{B^k(0,R)}\left(x_1+|x'-y'|\right)^{-(N-2\am-2)}\,\dd \nu(y')\right)^{p+q} \, \dd x_1\, \dd x'\\
	&\gtrsim 
	\int_{\mathbb{R}^k}\left(\sup_{0<r<\frac{R}{2}}\frac{\xn(B^k(x',r))}{r^{k-\vartheta}}\right)^{p+q}\, \dd x' \gtrsim \int_{\mathbb{R}^k}\BB_{k,\vartheta}[\xn](x')^{p+q} \, \dd x',
	\eal
which implies the lower bound in \eqref{est-Bnu0}.
\end{proof}

\begin{proof}[{\textbf{Proof of Theorem \ref{th:existence-source-4}}}] By virtue of Theorem \ref{existth}, it is sufficient to show that
\ba\label{Cap-equi-1}
\mathrm{Cap}_{\BBN_{2\am+1,1},(p+q)'}^{p+1,-\am(p+1)-(\am+1)q}(E)\asymp\mathrm{Cap}_{\vartheta,(p+q)'}^{\xS}(E)\quad\text{for any Borel set}\;E\subset\xS,
\ea	
where $\vartheta$ is defined in \eqref{vartheta}. From \eqref{Capsub}, we see that
\bal
\mathrm{Cap}_{\vartheta,p'}^{\xS}(E):=\sum_{i=1}^m \mathrm{Cap}_{\CB_{k,\vartheta},p'}^{\mathbb{R}^k}(\tl T_i(E \cap K_i)),
\eal
where $T_i(E \cap K_i)=\tl T_i(E \cap K_i) \times  \{ x'' = 0_{\mathbb{R}^{N-k}} \}$. With $\beta=(p+1)\am+q(\am+1)$, we have
\bal
\mathrm{Cap}_{\BBN_{2\am+1,1},(p+q)'}^{p+1,-\xb}(E)\asymp\sum_{i=1}^m\mathrm{Cap}_{\BBN_{2\am+1,1},(p+q)'}^{p+1,-\xb}(E\cap K_i).
\eal
Therefore, in order to prove \eqref{Cap-equi-1}, it's enough to show that
\ba \label{Cap-split} \mathrm{Cap}_{\CB_{k,\vartheta},(p+q)'}^{\mathbb{R}^k}(\tl T_i(E \cap K_i))\asymp \mathrm{Cap}_{\BBN_{2\am+1,1},(p+q)'}^{p+1,-\xb}(E \cap K_i), \quad i=1,2,\ldots,m.
\ea

Let $\xl \in \GTM^+(\partial\xO\cup \xS)$ with compact support on $\xS$ be such that $\BBK_\xm[\xl]\in L^p(\xO;\ei)$. Put $\lambda_{K_i} = \1_{K_i}\lambda$. On one hand, from \eqref{eigenfunctionestimates}, \eqref{Martinest1} and since $\xb<N-k$, we have
\bal
\int_{O_i}\BBK_\xm[\xl_{K_i}]^{p+q}(dd_\xS)^{-q}\ei \,\dx\gtrsim \lambda(K_i)^{p+q}\int_{O_i}d(x)^{p+1}d_\Sigma(x)^{-\xb} \,\dx\gtrsim\lambda(K_i)^{p+q}.
\eal
On the other hand,
\bal
\int_{\xO\setminus O_i}\BBK_\xm[\lambda_{K_i}]^{p+q}(dd_\xS)^{-q}\ei \dx\lesssim \lambda(K_i)^{p+q}\int_{\xO}d(x)^{p+1}d_\Sigma(x)^{-\xb}\dx\lesssim \lambda(K_i)^{p+q}.
\eal
Combining the above estimates, we derive
\ba\label{45}
\int_\xO\BBK_\xm[\lambda_{K_i}]^p (dd_\xS)^{-q}\ei \,\dx\asymp  \int_{O_i}\BBK_\xm[\lambda_{K_i}]^p(dd_\xS)^{-q}\ei \,\dx, \quad \forall i=1,2,..,m.
\ea

In view of the proof of \cite[Lemma 5.2.2]{Ad}, there exists a measure $\overline \lambda_i\in \GTM^+(\mathbb{R}^k)$ with compact support in $B^k(0,1)$ such that for any Borel $E\subset B^k(0,1)$, there holds
	\bal
\overline \lambda_i(E)=\lambda(T_i^{-1}(E\times \{0_{\R^{N-k}}\})).
\eal
Set $\psi=(\psi',\psi'')=T_i(x)$, we infer from \eqref{propdist}, \eqref{eigenfunctionestimates} and \eqref{Martinest1} that
	\bal
	&\ei(x)\asymp |\psi''|^{-\am},\\
	 &K_{\mu}(x,y)\asymp |\psi''|^{-\am}(|\psi''|+|\psi'-y'|)^{-(N-2\am-2)}, \quad
	 \forall x\in  O_i\setminus \Sigma,\;\forall y\in O_i\cap \Sigma.
	\eal
The above estimates, together with \eqref{45}, imply
\ba \label{46} \BAL
	&\int_{\Omega} \BBK_{\mu}[\lambda_{K_i}]^{p+q}(dd_\xS)^{-q}\ei \,\dx \asymp \int_{ O_i } \BBK_{\mu}[\lambda_{K_i}]^{p+q}(dd_\xS)^{-q}\ei \,\dx\\
	&\asymp \int_{B^k(0,1)}\int_{B^{N-k}(0,1)}|\psi''|^{-\xb}
	\left(\int_{B^k(0,1)}(|\psi''|+|\psi'-y'|)^{-(N-2\am-2)}\dd\overline{\lambda_i}(y')\right)^{p+q} \dd \psi'' \dd \psi'\\
	&=C(N,k)\int_{B^k(0,1)}\int_{0}^{\xb_0}r^{N-k-1-\xb}
	\left(\int_{B^k(0,1)}(r+|\psi'-y'|)^{-(N-2\am-2)}\dd\overline{\lambda_i}(y')\right)^{p+q} \dd r \dd\psi'.\\
    &\asymp \int_{\mathbb{R}^k} \BB_{k,\vartheta}[\overline \lambda_i](x')^p\dx'.
	\EAL
	\ea
Here the last estimate follows from \eqref{est-Bnu0}. From \eqref{KN2a1} and \eqref{46}, we deduce
\bal
\| \BBN_{2\am+1,1}[\xl_{K_i}]  \|_{L^{p+q}(\Omega; d^{p+1} d_\Sigma^{-\xb})}  \asymp \| \BBK_\xm[\xl_{K_i}] \|_{L^{p+q}(\Omega;(dd_\xS)^{-q}\phi_\mu)}  \asymp \| \BBB_{k,\vartheta}[\overline \lambda_i] \|_{L^{p+q}(\R^k)}.
\eal
This and \eqref{dualcap} lead to \eqref{Cap-split}, which in turn implies \eqref{Cap-equi-1}. The proof is complete.
\end{proof}

\begin{proof}[{\textbf{Proof of Theorem \ref{th:existence-source-5}}}]  By virtue of Theorem \ref{existth}, it is sufficient to show that
\ba\label{Cap-equi-2}
\mathrm{Cap}_{\BBN_{2\am+1,1},(p+q)'}^{p+1,-\am(p+1)-(\am+1)q}(E)\asymp\mathrm{Cap}_{\frac{2-q}{p+q},(p+q)'}^{\partial\xO}(E)\quad\text{for any Borel set}\;E\subset\partial\xO.
\ea	
	
Let $\beta=(p+1)\am+q(\am+1)$. By a similar argument as in the proof of \eqref{45}, by the assumptions on $p,q$, we can show that
for any $\lambda \in \GTM^+(\partial \Omega \cup \Sigma)$ with compact support on $\partial \Omega$, there holds
\bal
\int_\xO\BBK_\xm[\lambda]^{p+q}(dd_\xS)^{-q}\ei \dx\asymp \sum_{i=1}^m \int_{O_i}\BBK_\xm[ \1_{K_i} \lambda]^{p+q}(dd_\xS)^{-q}\ei \dx.
\eal
This and the estimate \eqref{KN2a1} imply
\bal \begin{aligned}
\int_{\Omega} \BBN_{2\am+1,1}[\lambda]^{p+q} d^{p+1} d_\Sigma^{-\xb}\dx & \asymp \sum_{i=1}^m \int_{O_i} \BBN_{2\am+1,1}[\1_{K_i} \lambda]^{p+q} d^{p+1} d_\Sigma^{-\xb} \dx \\
& \asymp \sum_{i=1}^m \int_{O_i} \BBN_{2\am+1,1}[\1_{K_i} \lambda]^{p+q} d^{p+1} \dx.
\end{aligned} \eal
Therefore, in view of the proof of \cite[Proposition 2.9]{BHV} (with $\xa=\xb=1$, $s=(p+q)'$ and $\xa_0=p+1$) and \eqref{dualcap}, we obtain \eqref{Cap-equi-2}. The proof is complete.
\end{proof}

\subsection{The case $\xS=\{0\}$ and $\xm=(\frac{N-2}{2})^2$} 
In this case $\xm= H^2=(\frac{N-2}{2})^2$, $\am=\frac{N-2}{2}$ and $\CN_{\alpha,2}$ defined in \eqref{Nas} becomes
\bal
\CN_{\alpha,2}(x,y):=\frac{\max\{|x-y|,|x|,|y|\}^{\alpha}}{|x-y|^{N-2}\max\{|x-y|,d(x),d(y)\}^2},\quad\forall(x,y)\in\overline{\xO}\times\overline{\xO}, x \neq y.
\eal

For $0<\xe<N-2$, we introduce some additional functions as follows
\bal
\CM_{\varepsilon}(x,y):=\frac{\max\{|x-y|,|x|,|y|\}^{N-2}+|x-y|^{N-2-\xe}}{|x-y|^{N-2}\max\{|x-y|,d(x),d(y)\}^2},\quad\forall(x,y)\in\overline{\xO}\times\overline{\xO}, x \neq y,	
\eal
\bal \begin{aligned}
G_{H^2,\xe}(x,y) &:= |x-y|^{2-N} \left(1 \wedge \frac{d(x)d(y)}{|x-y|^2}\right) \left(1 \wedge \frac{|x||y|}{|x-y|^2} \right)^{-\frac{N-2}{2}} \\
&\quad +(|x||y|)^{-\frac{N-2}{2}}|x-y|^{-\xe}  \left(1 \wedge \frac{d(x)d(y)}{|x-y|^2}\right), \quad x,y \in \Omega \setminus \{0\}, \, x \neq y,
\end{aligned} \eal
\bal
\tilde G_{H^2,\xe}(x,y):=d(x)d(y)(|x||y|)^{-\frac{N-2}{2}}  \CN_{N-2-\varepsilon,2}(x,y), \quad \forall x,y \in \Omega \setminus \{0\}, \, x \neq y.
\eal

We see that
\bal \begin{aligned}
(|x||y|)^{-\frac{N-2}{2}}\left|\ln\left(1 \wedge \frac{|x-y|^2}{d(x)d(y)}\right)\right|
&\leq(|x||y|)^{-\frac{N-2}{2}}\left|\ln\frac{|x-y|}{\mathcal{D}_\Omega}\right|\left(1 \wedge \frac{d(x)d(y)}{|x-y|^2}\right) \\
&\leq C(\Omega,\xe) (|x||y|)^{-\frac{N-2}{2}}|x-y|^{-\xe}  \left(1 \wedge \frac{d(x)d(y)}{|x-y|^2}\right),
\end{aligned}
\eal
which, combined  with \eqref{Greenestb}, yields
\bal
G_{H^2}(x,y)\lesssim G_{H^2,\xe}(x,y), \quad \forall x,y \in \Omega \setminus \{0\}, \, x \neq y.
\eal

Next, we also observe that
\bal
G_{H^2,\varepsilon}(x,y) \asymp d(x)d(y)(|x||y|)^{-\frac{N-2}{2}} \CM_{\varepsilon}(x,y), \quad x,y \in \Omega \setminus \{0\}, \, x \neq y,
\eal
\bal
\CM_{\xe}(x,y)\leq C(\xe,\xO) \CN_{N-2-\varepsilon,2}(x,y),\quad x,y \in \Omega \setminus \{0\}, \, x \neq y,
\eal
hence
\ba \label{GtildeG}
G_{H^2,\xe}(x,y)\lesssim \tilde G_{H^2,\xe}(x,y), \quad \forall x,y \in \Omega \setminus \{0\}, \, x \neq y.
\ea
By combining estimates \eqref{GtildeG}, \eqref{Martinest1-b}, \eqref{gradientgreen}, \eqref{est:gradK-K} and the estimate $$\CN_{N-2-\xe,2}(x,y) \lesssim \CN_{N-1-\xe,1}(x,y), \quad \forall x,y \in \Omega \setminus \{0\}, \, x \neq y,$$ 
we derive the following bounds
\bal \BAL
G_{H^2}(x,y)\lesssim d(x)d(y)(|x||y|)^{-\frac{N-2}{2}} \CN_{N-1-\xe,1}(x,y) \quad \forall x,y \in \Omega \setminus \Sigma, \, x \neq y,
\EAL \eal 
\bal 
|\nabla_x G_{H^2}(x,y)|\lesssim |x|^{-\frac{N}{2}}d(y)|y|^{-\frac{N-2}{2}} \CN_{N-1-\xe,1}(x,y) \quad \forall x,y \in \Omega \setminus \{0\}, \, x \neq y,
\eal
\bel{KN2a2}
K_{H^2}(x,z)\lesssim d(x)|x|^{-\frac{N-2}{2}} \CN_{N-1-\xe,1}(x,z) \quad \forall x \in \Omega \setminus \{0\}, \, z \in \partial\xO\cup \{0\},
\ee
and
\bal 
|\nabla_x K_{H^2}(x,z)|\lesssim |x|^{-\frac{N}{2}} \CN_{N-1-\xe,1}(x,z) \quad \forall x \in \Omega \setminus \{0\}, \, z \in \partial\xO\cup\{0\}.
\eal 
By applying Proposition \ref{t2.1} with $J(x,y)=\CN_{N-1-\xe,1}(x,y)$, $\dd \xo= d(x)^{p+1}|x|^{-\frac{N-2}{2}(p+1)-\frac{N}{2}q}\dx$ and $\lambda = \xn$, we obtain the following result

\begin{theorem}\label{sourceth2} Assume  $\xn\in\mathfrak{M}^+(\partial\xO\cup \{0\})$, $p \geq 0$, $q \geq 0$ such that $p+q>1$ and $(N-2)p+Nq<N+2$. Let $0<\varepsilon <\min(N-2,2,\frac{N+2- (N-2)p - Nq}{2(p+q)})$. Then the following statements are equivalent.
	
	1. For $\ell>0$, the following equation has a positive solution
\bal
v=\BBN_{N-1-\xe,1}[|v|^{p+q}d(\cdot)^{p+1} 
|\cdot|^{-\frac{N-2}{2}(p+1)-\frac{N}{2}q}]+\ell \BBN_{N-1-\xe,1}[\xn].
\eal
	
	2. For any Borel set $E \subset \overline{\xO}$, there holds
\ba\label{con1cr}
\int_E \BBN_{N-1-\xe,1}[\1_E\xn](x)^{p+q} d(x)^{p+1}|x|^{-\frac{N-2}{2}(p+1)-\frac{N}{2}q}\dx \leq C\, \xn(E).
\ea
	
	3. The following inequality holds
\bal
\BBN_{N-1-\xe,1}\bigg[\BBN_{N-1-\xe,1}[\xn]^{p+q}d(\cdot)^{p+1}|\cdot|^{-\frac{N-2}{2}(p+1)-\frac{N}{2}q}\bigg]\leq C\BBN_{N-1-\xe,1}[\xn]<\infty\quad a.e.
\eal

	4. For any Borel set $E \subset \overline{\xO}$, there holds
\bal 
\xn(E)\leq C\, \mathrm{Cap}_{\BBN_{N-1-\xe,1},(p+q)'}^{p+1,-\frac{N-2}{2}(p+1)-\frac{N}{2}q}(E).
\eal
\end{theorem}


\begin{proof}[{\textbf{Proof of Theorem \ref{existthcr}}}] By using a similar argument as in the the proof of Theorem \ref{existth}, making use of Theorem \ref{sourceth2} instead of Theorem \ref{sourceth}, we can obtain the desired result.   
\end{proof}

\begin{proof}[{\textbf{Proof of Theorem \ref{th:existence-H2-2}}}] In view of Theorem \ref{sourceth2} and Theorem \ref{existthcr}, it is enough to show that condition \eqref{con1cr} holds.

(i) Assume $\nu=\delta_0$. Let $E\subset\overline{\xO}$ be a Borel set such that $0\in E$. Using the definition of $\CN_{N-1-\xe,1}$ in \eqref{Nas} and the condition on $\varepsilon$, we have
\bal 
\int_E \BBN_{N-1-\xe,1}[\1_E\xn](x)^{p+q} d(x)^{p+1}|x|^{-\frac{N-2}{2}(p+1)-\frac{N}{2}q}\dx 
\lesssim \int_\xO |x|^{-\frac{N-2}{2}(p+1)-\frac{N}{2}q-\xe(p+q)}\dx\lesssim 1.
\eal

(ii) Assume $\nu$ has compact support on $\partial \Omega$. Let $E\subset\overline{\xO}$ be a Borel. By \eqref{KN2a2}, we have
\ba\BAL\label{1cr}
&\int_E \BBN_{N-1-\xe,1}[\1_E\xn](x)^{p+q} d(x)^{p+1}|x|^{-\frac{N-2}{2}(p+1)-\frac{N}{2}q}\dx \\
&\asymp \int_E d(x)|x|^{-\frac{N-2}{2}}\left((d(x)|x|)^{-\frac{q}{p+q}}\mathbb{K}_\mu[\1_E\xn]\right)^{p+q}\,\dx \\
&\leq \int_\xO d(x)|x|^{-\frac{N-2}{2}}\left((d(x)|x|)^{-\frac{q}{p+q}}\mathbb{K}_\mu[\1_E\xn]\right)^{p+q}\,\dx.
\EAL
\ea

 By \eqref{estgranmartin1} with $\xg=\frac{q}{p+q},$ we have
\bal 
\int_\xO d(x)|x|^{-\frac{N-2}{2}}\left((d(x)|x|)^{-\frac{q}{p+q}}\mathbb{K}_\mu[\1_E\xn](x)\right)^{p+q}\,\dx\leq C \xn(E \cap \partial \Omega)^{p+q}\leq C\xn(\partial\xO)^{p+q-1}\xn(E).
\eal 
The desired result follows by the above inequality and \eqref{1cr}.
\end{proof}

\begin{proof}[{\textbf{Proof of Theorem \ref{th:existence-H2-3}}}]
By virtue of Theorem \ref{sourceth2}, it is enough to show that 
\ba \label{cap-equi-H2}
\mathrm{Cap}_{\BBN_{N-1-\xe,1},(p+q)'}^{p+1,-\frac{N-2}{2}(p+1)-\frac{N}{2}q}(E)\asymp\mathrm{Cap}_{\frac{2-q}{p+q},(p+q)'}^{\partial\xO}(E)\quad\text{for any Borel set}\;E\subset\partial\xO.
\ea
Indeed, since $\xn\in \mathfrak{M}^+(\partial\xO\cup \{0\})$ has compact support on $\partial\xO$, it follows that $\BBN_{N-1-\xe,1}[\xn]\asymp\BBN_{0,1}[\xn]$. In view of the proof of \eqref{Cap-equi-2}, we obtain \eqref{cap-equi-H2}. The proof is complete.
\end{proof}

\appendix

\section{Properties of the first eigenpair of $-L_\mu$} \label{subsect:eigen} Let $H$ be the constant defined in \eqref{valueH} and for $\mu \leq H^2$, let $\am$ and $\ap$ be defined in \eqref{apm}. Important properties of the first eigenvalue and its corresponding eigenfunction associated to $-L_\mu$ in $\Omega \setminus \Sigma$ are presented below (see e.g. \cite[Lemma 2.4 and Theorem 2.6]{DD1}).

(i) For any $\mu \leq H^2$, it is known that
\be\label{Lin01} \lambda_\mu:=\inf\left\{\int_{\Gw}\left(|\nabla u|^2-\frac{\xm }{d_\Sigma^2}u^2\right)\dx: u \in C_c^1(\Omega), \; \|u\|_{L^2(\Omega)}=1\right\}>-\infty.
\ee

\smallskip

(ii) If $\mu < H^2$, there exists a minimizer $\gf_{\xm }$ of \eqref{Lin01} belonging to $H^1_0(\Gw)$ and satisfying $-L_\mu \phi_\mu= \lambda_\mu \phi_\mu$  in $\Omega \setminus \Sigma$. Moreover,
$\phi_{\mu }\asymp d_\Sigma^{-\am}$ in $\Sigma_{\beta_0}$.

\smallskip

(iii) If $\xm =H^2$, there is no minimizer of \eqref{Lin01} in $H_0^1(\Gw)$, but there exists a nonnegative function $\phi_{H^2}\in H_{\mathrm{loc}}^1(\xO)$  such that $-L_{H^2}\phi_{H^2}=\lambda_{H^2}\phi_{H^2}$ in the sense of distributions in $\Omega \setminus \Sigma$ and $\phi_{H^2}\asymp d_\Sigma^{-H}  \quad  \text{in } \Sigma_{\beta_0}$.
In addition, $d_\Sigma^{-H}\xf_{H^2}\in H^1_0(\Gw; d_\Sigma^{-2H})$.

From (ii) and (iii) we deduce that, for $\mu \leq H^2$, there holds
\be \label{eigenfunctionestimates}
\xf_\xm \asymp d\,d^{-\am}_\Sigma \quad \text{in } \Omega \setminus \Sigma.
\ee

\section{$L_\mu$-harmonic measures}\label{appendix:har-measure}
Let $\beta_0$ be the constant in Subsection \ref{assumptionK}. Let $\eta_{\beta_0}$ be a smooth function such that $0 \leq \eta_{\beta_0} \leq 1$, $\eta_{\beta_0}=1$ in $\overline{\Sigma}_{\frac{\xb_0}{4}}$ and $\supp \eta_{\beta_0} \subset \Sigma_{\frac{\beta_0}{2}}$. We define
\bal W(x):=\left\{ \BAL &d_\Sigma(x)^{-\ap} \qquad&&\text{if}\;\mu <H^2, \\
&d_\Sigma(x)^{-H}|\ln d_\Sigma(x)| \qquad&&\text{if}\;\mu =H^2,
\EAL \right. \quad x \in \Omega \setminus \Sigma,
\eal
and
\bal
\tilde W(x):=1-\eta_{\beta_0}(x)+\eta_{\beta_0}(x)W(x), \quad x \in \Omega \setminus \Sigma.
\eal

Let $z \in \Omega \setminus \Sigma$ and $h\in C(\partial\Omega \cup \Sigma)$ and denote $\CL_{\mu ,z}(h):=v_h(z)$ where $v_h$ is the unique solution of the Dirichlet problem
\be \label{linear} \left\{ \BAL
L_{\mu}v&=0\qquad \text{in}\;\;\xO\setminus \Sigma\\
v&=h\qquad \text{on}\;\;\partial\xO\cup \Sigma.
\EAL \right. \ee
Here the boundary value condition in \eqref{linear} is understood in the sense that
\bal
\lim_{\dist(x,F)\to 0}\frac{v(x)}{\tilde W(x)}=h \quad \text{for every compact set } \; F\subset \partial \Omega \cup \Sigma.
\eal
The mapping $h\mapsto \CL_{\mu,z}(h)$ is a linear positive functional on $C(\partial\Omega \cup \Sigma)$. Thus there exists a unique Borel measure on $\partial\Omega \cup \Sigma$, called {\it $L_{\mu}$-harmonic measure in $\partial \Omega \cup \Sigma$ relative to $z$} and  denoted by $\omega_{\Omega \setminus \Sigma}^{z}$, such that
\bal
v_{h}(z)=\int_{\partial\Omega\cup \Sigma}h(y) \dd\omega_{\Omega \setminus \Sigma}^{z}(y).
\eal
Let $x_0 \in \Omega \setminus \Sigma$ be a fixed reference point. Let $\{\xO_n\}$ be an increasing sequence of bounded $C^2$ domains  such that
\bal
\overline{\xO}_n\subset \xO_{n+1}, \quad \cup_n\xO_n=\xO, \quad \mathcal{H}^{N-1}(\partial \Omega_n)\to \mathcal{H}^{N-1}(\partial \Omega),
\eal
where $\mathcal{H}^{N-1}$ denotes the $(N-1)$-dimensional Hausdorff measure in $\R^N$.
Let $\{\Sigma_n\}$ be a decreasing sequence of bounded $C^2$ domains  such that
\bal \Sigma \subset \Sigma_{n+1}\subset\overline{\Sigma}_{n+1}\subset \Sigma_{n}\subset\overline{\Sigma}_{n} \subset\Omega_n, \quad \cap_n \Sigma_n=\Sigma.
\eal
For each $n$, set $O_n=\xO_n\setminus \Sigma_n$  and assume that $x_0 \in O_1$. Such a sequence $\{O_n\}$ will be called a {\it $C^2$ exhaustion} of $\Gw\setminus \Sigma$. Then $-L_\mu$ is uniformly elliptic and coercive in $H^1_0(O_n)$ and its first eigenvalue $\lambda_\mu^{O_n}$ in $O_n$ is larger than its first eigenvalue $\lambda_\mu$ in $\Omega \setminus \Sigma$.

For $h\in C(\prt O_n)$, the following problem
\bal\left\{ \BAL
-L_{\xm } v&=0\qquad&&\text{in } O_n\\
v&=h\qquad&&\text{on } \prt O_n,
\EAL \right.
\eal
admits a unique solution which allows to define the $L_{\xm }$-harmonic measure $\omega_{O_n}^{x_0}$ on $\prt O_n$
by
\bal
v(x_0)=\myint{\prt O_n}{}h(y) \dd\gw^{x_0}_{O_n}(y).
\eal
The $L_{\xm }$-harmonic measures $\omega_{O_n}^{x_0}$ on $\prt O_n$ are used to define the boundary trace which is defined in a \textit{dynamic way} as in Definition \ref{nomtrace}.
%
%


\end{document}